\documentclass{amsart}
\usepackage{amsmath}
\usepackage{amsthm}
\usepackage{amssymb}
\usepackage{color}
\usepackage{enumerate}

\newtheorem{thm}{Theorem}
\newtheorem{lem}[thm]{Lemma}

\newtheorem{prop}[thm]{Proposition}
\newtheorem{rk}[thm]{Remark}
\newtheorem{nota}[thm]{Notation}

\newcommand{\R}{{\mathbb{R}}}

\newcommand{\Z}{{\mathbb{Z}}}
\newcommand{\cF}{{\mathcal{F}}}
\newcommand{\cU}{{\mathcal{U}}}
\newcommand{\cZ}{{\mathcal{Z}}}
\newcommand{\Law}{{\rm Law}}
\newcommand{\cL}{{\mathcal{L}}}
\newcommand{\cE}{{\mathcal{E}}}
\newcommand{\e}{\varepsilon}
\newcommand{\vip}{\vskip.15cm}
\newcommand{\indiq}{{\bf 1}}
\newcommand{\E}{\mathbb{E}}
\newcommand{\PR}{\mathbb{P}}
\newcommand{\intot}{\int_0^t }
\newcommand{\intrd}{\int_{\R^d}}
\newcommand{\intMK}{\int_{M_K} }
\newcommand{\dd}{{\rm d}}
\newcommand{\ddiv}{{\rm div}}

\newcommand{\ttau}{{\tilde \tau}}

\newcommand{\bB}{{\bar B}}
\newcommand{\dB}{{\partial B}}

\begin{document}

\title{On the simulated annealing in $\R^d$}

\author{Nicolas Fournier and Camille Tardif}

\address{Nicolas Fournier and Camille Tardif : Sorbonne Universit\'e, 
LPSM-UMR 8001,
Case courrier 158, 75252 Paris Cedex 05, France.}
\email{nicolas.fournier@sorbonne-universite.fr, camille.tardif@sorbonne-universite.fr}


\thanks{We warmly thank Pierre Monmarch\'e for fruitful discussions}



\begin{abstract}
Using a localization procedure and the result of 
Holley-Kusuoka-Stroock \cite{hks} in the torus, 
we widely weaken the usual growth assumptions
concerning the success of the continuous-time simulated annealing in $\R^d$.
Our only assumption is the existence of an invariant probability measure
for a sufficiently low temperature.
We also prove, in an appendix, a non-explosion criterion for a class of time-inhomogeneous diffusions.
\end{abstract}

\maketitle

\section{Introduction and results}

\subsection{Main results}

We work with the following setting.

\vip

\noindent {\bf Assumption} $(A)$. {\it Fix a dimension $d\geq 1$
and a function $U:\R^d \to \R_+$ of class $C^\infty$ such that 
$\lim_{|x|\to \infty} U(x)=\infty$ and $\min_{x\in \R^d} U(x)=0$. 
For $x,y\in \R^d$, we set 
$$
E(x,y)=\inf \Big\{\max_{t\in [0,1]}U(\gamma_t)-U(x)-U(y) \;:\; 
\gamma \in C([0,1],\R^d), \gamma_0=x,\gamma_1=y \Big\}
$$
and we suppose that $c_*=\sup\{E(x,y) : x,y\in\R^d\}<\infty$.}

\vip

Actually, $c_*=\sup\{E(x,y):x$ local minimum of $U$, $y$ global minimum of $U\}$
represents the maximum potential energy required to reach a global minimum $y$ of $U$ when starting 
from {anywhere else.}

\vip

We fix $x_0 \in \R^d$, $c>0$ and $\beta_0>0$
and consider the time-inhomogeneous S.D.E.
\begin{equation}\label{eds}
X_t = x_0 + B_t - \frac 12\intot \beta_s \nabla U(X_s) \dd s \quad \hbox{where} 
\quad \beta_t=\frac{\log (e^{c \beta_0}+t)}{c}
\end{equation}
and where $(B_t)_{t\geq0}$ is a $d$-dimensional Brownian motion.
By Theorem \ref{nonex} proved in the appendix, since $U\geq 0$ under $(A)$, 
\eqref{eds} has a pathwise unique non-exploding solution $(X_t)_{t\geq 0}$.
Here is our main result.

\begin{thm}\label{main}
Assume $(A)$ and that $\intrd e^{-\alpha_0 U(x)}\dd x <\infty$ 
for some $\alpha_0>0$. Fix $c>c_*$, $x_0 \in \R^d$ and $\beta_0 >0$ and consider the unique solution
$(X_t)_{t\geq 0}$ to \eqref{eds}. Then
$\lim_{t\to \infty}U(X_t)= 0$ in probability.
\end{thm}

One of the ingredients of the proof is the following proposition, which asserts that, in full generality,
the simulated annealing is successful on the event where the process $(X_t)_{t\geq 0}$
does not escape to infinity in large time.

\begin{prop}\label{mm}
Assume $(A)$, fix $c>c_*$, $x_0\in \R^d$ and $\beta_0 >0$.
For $(X_t)_{t\geq 0}$ the solution to \eqref{eds},
$$
\forall \; \e>0,\quad 
\lim_{t\to \infty}  \PR\Big(\liminf_{s\to\infty} |X_s|<\infty  \hbox{ and } U(X_t)>\e\Big)=0.
$$
\end{prop}

\subsection{Comments and references}
The simulated annealing has been introduced by 
Kirkpatrick-Gelatt-Vecchi \cite{kgv} as a numerical procedure to find a 
(possibly non unique) 
global minimum of a function $U$ on a given state space. We refer to Azencott \cite{a} 
for an early review of the method and its links with the theory of Freidlin-Wentzell 
\cite{fw}.

\vip

With our notation and in our context where the state space is $\R^d$, 
the main idea of the simulated annealing is the following. 
The solution to \eqref{eds}, with $\beta$
constant, has $\mu_\beta(\dd x)=\cZ_\beta^{-1} e^{-\beta U(x)} \dd x$ as invariant 
probability distribution, if $\cZ_\beta=\intrd e^{-\beta U(x)} \dd x<\infty$.

\vip

Using that $\min U =0$ and that $\nabla U$ is locally bounded, we deduce 
that there is $\kappa>0$ such that $\cZ_\beta \geq e^{-1}\int_{\{U\leq 1/\beta\}} \dd x\geq
\kappa \beta^{-d}$ for all $\beta\geq 1$.
Hence under the condition that $\cZ_{\alpha_0}<\infty$ for some $\alpha_0>0$, it
holds that for all $\e>0$, 
\begin{equation}\label{pchc}
\mu_\beta(U>\e) \leq \cZ_\beta^{-1}\int_{\{U>\e\}}e^{-\beta U(x)}\dd x \leq  \frac{\beta^d}\kappa
\cZ_{\alpha_0} e^{-(\beta-\alpha_0) \e} \longrightarrow 0\quad \hbox{as $\beta\to\infty$.}
\end{equation}
Hence one hopes that the solution to \eqref{eds},
with $\lim_{t\to \infty} \beta_t=\infty$, satisfies $\lim_{t\to \infty} U(X_t)=0$ in probability.
However, it is necessary that $\beta_t$ 
increases sufficiently slowly to infinity, so that $\Law(X_t)$ remains close, 
for all times, to $\mu_{\beta_t}$. If $\beta_t$ increases too fast to infinity, one may
remain stuck near a local minimum of $U$, as in the classical 
deterministic gradient method.

\vip

A major contribution is due to Holley-Kusuoka-Stroock \cite{hks}, see also
Holley-Stroock \cite{hs}. Replacing
$\R^d$ by a compact manifold $M$, they showed that when 
$\beta_t \simeq c^{-1}\log (1+t)$,
the simulated annealing procedure
is successful, i.e. $\lim_{t\to \infty} U(X_t)=0$ in probability, if and only if $c>c_*$.
Their proof is almost purely analytic and very elegant.
It relies on precise spectral gap estimates providing 
an asymptotically optimal Poincar\'e inequality.
They use at many places the compactness of the state space.

\vip

This kind of proof involving functional inequalities
has been extended to the non-compact case of $\R^d$ by Royer \cite{r}
and Miclo \cite{mic}, at the price of many growth conditions on $U$, like
\begin{equation}\label{bg}
\lim_{|x|\to \infty} U(x)= \lim_{|x|\to \infty} |\nabla U(x)|=\infty \quad \hbox{and}
\quad \forall \; x\in\R^d, \quad \Delta U(x) \leq C + |\nabla U(x)|^2.
\end{equation}
Zitt \cite{z}, taking advantage of some weak Poincar\'e inequalities,
worked under another set of rather stringent conditions, still implying that
all the local minima of $U$ are lying in a compact set.
He in particular assumes that $|\nabla U|$ is bounded and that there is $\e>0$
such that, for all $x$ outside a compact, $U(x) \geq \log^{1+\e} |x|$ and $\Delta U(x) \leq 0$.
\vip

Here we only assume that $\intrd e^{-\alpha_0 U(x)}\dd x<\infty$ for some $\alpha_0>0$,
which seems very natural in view of \eqref{pchc}.
This covers and consequently extends the previously cited works in $\R^d$. 
In particular, nothing forbids $U$ to oscillate, as strongly as it wants,
and as far as it wants from compact sets, and thus in particular
to have an unbounded set of local minima.

\subsection{Short heuristics}
Let us emphasis that our proof
relies on the following two main points.
An entropy computation,
see Lemma \ref{super}, shows that the condition
$\intrd e^{-\alpha_0 U(x)}\dd x <\infty$ implies that $\liminf_{t\to \infty} |X_t|<\infty$ a.s.
Now, recall that in the compact case, see Holley-Kusuoka-Stroock \cite{hks} or Miclo \cite{mic2},
its a.s. holds that $\limsup_{t\to \infty} U(X_t)=c$.
Combining these two points, its seems rather clear from the Borel-Cantelli Lemma that,
in the non compact setting, the process will also satisfy
$\limsup_{t\to \infty} U(X_t)=c$ a.s. Hence it will eventually remain in a compact set
and Theorem \ref{main} will follow from the compact case.

\vip

Apart from Lemma \ref{super}, which seems new and efficient,
there are a number of technical issues, that are detailed in the next subsection.

\subsection{Plan of the proof}

We denote by $(X_t)_{t\geq 0}$ the solution to \eqref{eds}. We 
assume $(A)$ and the conditions that 
$\intrd e^{-\alpha_0 U(x)}\dd x<\infty$ for some $\alpha_0>0$ and $c>c_*$.

\vip

(a) In Section \ref{pr}, we prove some auxiliary weak regularization 
property for
the law of the solution to \eqref{eds}. This allows us, when applying P.D.E. techniques,
to do as if the law of $X_0$ had a bounded density concentrated around $x_0$,
with a precise bound as a function of $\beta_0$.

\vip

(b) In Section \ref{imp1}, we show that
$\liminf_{t\to \infty} |X_t|<\infty$: the process cannot escape to infinity in large time.
This does not use the condition $c>c_*$.
The key argument is the following:
under the additional assumptions that 
$\Law(X_0)$ is smooth and $\beta_0>\alpha_0$, we prove
the important {\it a priori} 
estimate $\sup_{t\geq 0} \E[U(X_t)]<\infty$, see Lemma \ref{super}, which {\it a priori} implies that
$\liminf_{t\to \infty} |X_t|<\infty$ by the Fatou lemma and since $\lim_{|x|\to\infty} U(x)=\infty$.
We then make all this rigorous and get rid of the additional assumptions using point (a) and that
our process does not explode in finite time.

\vip

This central {\it a priori} estimate is derived from a rather original entropy computation.
Let us mention that deducing that $\sup_{t\geq 0} \E[U(X_t)]<\infty$ directly from the It\^o
formula would necessarily require some stringent conditions on $\nabla U$ and $\Delta U$. 

\vip

(c) In Section \ref{imp2}, we verify in Lemma \ref{fmt} that,
with an abuse of language, $U(X_t)\to 0$ 
in probability as $t\to \infty$ on the event where $\sup_{t\geq 0} |X_t|<\infty$.

\vip

This is easy, by localization, in view of the results of 
Holley-Kusuoka-Stroock \cite{hks} applied to a large flat torus: the condition
$\sup_{t\geq 0} |X_t|<\infty$ almost tells us that we are in a compact setting.

\vip

(d) Still in Section \ref{imp2}, we check, although stated in slightly different words, 
see Proposition \ref{crucial}, that for any $B\geq 1$, there are 
$C_B>B$ and $t_B>0$ such that 
$$
\inf_{|x_0|\leq B, t_0\geq t_B} \PR_{t_0,x_0}\Big(\sup_{t\geq 0} |X_t|\leq C_B\Big) \geq \frac12.
$$
This is rather natural: in the compact setting, it is well-known, see \cite{hks} or Miclo
\cite{mic2}, that $\limsup_{t\to \infty} U(X_t)=c$ a.s. It would not be too difficult to
deduce that in the non-compact case, there exists $C_{t_0,x_0}>0$ such that
$\PR_{t_0,x_0}(\sup_{t\geq 0} |X_t|\leq C_{t_0,x_0}) \geq 1/2$.
The main issue is to show that $C_{t_0,x_0}$ does not depend too much on $t_0$ and
$x_0$. This is tedious, and we have to revisit the proof of \cite{hks}.

\vip

(e) In Section \ref{con}, we prove
Proposition \ref{mm}: by (d), on the event 
$\liminf_{s\to \infty} |X_s|<\infty$, our process will eventually be absorbed in a compact set, so that 
$\sup_{s\geq 0} |X_s|<\infty$, whence the success of the simulated annealing by point (c). 

\vip
(f) Still in Section \ref{con}, we conclude the proof of Theorem \ref{main}:
$\liminf_{t\to\infty} |X_t|<\infty$ a.s. by (b), whence the success of the simulated annealing by (e).

\subsection{More comments}

It is well-know that, even in the compact case,
the condition $c>c_*$ is necessary, see Holley-Kusuoka-Stroock \cite[Corollary 3.11]{hks}.

\vip

Our proof completely breaks down for slower freezing schemes, i.e. 
if $\beta_t\ll \log t$ as $t\to\infty$:
in such a case, point (d) above  cannot hold true, even non uniformly
in $t_0$ and $x_0$.

\vip

Observe that we do not assume any Lyapunov condition, which would involve $\Delta U$ and $\nabla U$
and would forbid $U$ to oscillate too strongly.

\vip

As already mentioned and in view of \eqref{pchc}, our only assumption, i.e. the existence of an invariant probability measure for some (low) temperature, is very natural and allows 
for potentials with a very general shape.

\vip

However, we have shown in a previous paper with Monmarch\'e \cite{fmt}
that things may work even without this condition. 
In \cite[Theorem 1 and Proposition 2]{fmt}, we see that if $d\geq 3$ and
$U(x)=a\log\log|x|$ outside a compact, the simulated annealing works if $c>c_*$ and $c<2a/(d-2)$
and fails if $c>2a/(d-2)$. But it is not clear that a general growth condition exists. 
In particular, we deduce from
\cite[Proposition 2]{fmt} and a comparison argument that if $U(x)=\log^{\circ 3} |x|$
outside a compact,
then the simulated annealing fails for all $c>0$. But in  
\cite[Proposition 3]{fmt}, we built some (very oscillating) potential $U$
such that $\log^{\circ 3}|x| \leq U(x) \leq 3\log^{\circ 3}|x|$ outside a compact for which the simulated 
annealing works for 
some values of $c$.
Thus, without the condition that $\intrd \exp(-\alpha_0 U(x))\dd x<\infty$
for some $\alpha_0>0$,
the situation may be very intricate and really depend on the shape of $U$.

\subsection{Non-explosion}
The non-explosion of the solution to \eqref{eds}, using only that $U\geq 0$, is
checked in the appendix.
Actually, we treat, without major complication, 
the more general case where $\beta:\R_+\to (0,\infty)$ is any smooth function
and where $U:\R^d\to \R$ is smooth and satisfies $U(x)\geq -L(1+|x|^2)$ for some constant $L>0$.
This is not so easy, since we do not want to assume any local condition on $\nabla U$.
We use purely deterministic techniques inspired by the seminal work of 
Grigor'yan \cite{gr1}, also exposed in \cite[Section 9]{gr2} and by the paper of
Ichihara \cite{ichi},
both dealing with more general but {\it time-homogeneous} processes.

\vip
 
Let us mention that in the homogeneous case,
Ichihara uses the P.D.E. satisfied by $v(x)=\E_x[e^{-\sigma_1}]$, where
$\sigma_1=\inf\{t\geq 0 : |X_t|\leq 1\}$, while Grigor'yan rather studies the P.D.E. satisfied by
$w(t,x)=\PR_{x}[\zeta<t]$, where $\zeta$ is the life-time of the solution. 
In the inhomogeneous setting, we study, roughly, the P.D.E. satisfied by
$u(t,x)=\E_{t,x}[e^{-\zeta}]$, where $\zeta$ is the
life-time of the solution. The situation is slightly more complicated, but we manage to take
advantage of some computations found in \cite{gr1} and \cite{ichi} to show that $u\equiv 0$.

\section{Weak regularization}\label{pr}

We prove some weak regularization that will allow us, when using P.D.E. 
techniques, to replace the Dirac initial condition $\delta_{x_0}$ by some bounded function
concentrated around $x_0$.
One might invoke the H\"ormander theorem,
but since we need a precise bound as a function of $\beta_0$ (see Lemma \ref{cs} below),
we will rather
use the following weaker lemma based on stopping times.

\begin{lem}\label{tau}
Assume $(A)$ and fix $c>0$.
For any $A>1$, there is a constant $C_A^{(1)}$ such that for any 
$x_0\in \{U\leq A\}$, any $\beta_0>0$, denoting by 
$(X_t)_{t\geq 0}$ the corresponding solution to \eqref{eds}, 
there exists a stopping time
$\tau \in [0,1]$ such that $\sup_{t\in[0,\tau]}|X_t - x_0|\leq 1$
and such that the law of $(\tau,X_\tau)$ has a density bounded by
$\exp(C_A^{(1)}(\beta_0+1))\indiq_{\{[0,1]\times B(x_0,1)\}}$.
\end{lem}

\begin{proof}
We fix $x_0 \in \{U\leq A\}$ and $\beta_0>0$.
We introduce some random variable $R$, uniformly distributed in $[1/2,1]$
and independent of $(X_t)_{t\geq 0}$. We claim that
$$
\tau = \inf\{t \geq 0 : |X_t - x_0|=R\} \land R
$$
satisfies the requirements of the statement. 

\vip

First, $\tau \leq R \leq 1$ and $\sup_{t\in [0,\tau]}|X_t - x_0|\leq R\leq 1$.

\vip

Next, we consider a $d$-dimensional Brownian motion $(W_t)_{t\geq 0}$ independent of $R$ and
we set $\ttau = \inf\{t \geq 0 : |W_t|=R\} \land R$.
We introduce
the martingale 
$$
L_t=-\frac{1}{2} \int_0^{t\land \ttau} \beta_s\nabla U(x_0+W_s) \cdot \dd W_s,
$$
as well as its exponential $\cE_t=\exp(L_t - \frac12 \langle L \rangle_t)$,
which is uniformly integrable by the Novikov criterion, see Revuz-Yor 
\cite[Proposition 1.15 p 332]{ry}, because 
$\langle L \rangle_\infty \leq \frac 1 4 (\sup_{B(x_0,1)} |\nabla U|^2)
\int_0^1 \beta_s^2 \dd s$ is bounded.
The Girsanov theorem tells us that under $\cE_\infty\cdot \PR$, the process
$$
B_{t\land \ttau} = W_{t\land\ttau}+\frac{1}{2} \int_0^{t\land \ttau} \beta_s\nabla U(x_0+W_s) \dd s,
\quad t\geq 0
$$
is a (stopped) Brownian motion, so that $x_0+W_{t\land\ttau}$ is a (stopped) solution
to \eqref{eds}. Hence for all measurable
$\phi: \R_+\times
\R^d \to \R_+$, $\E[\phi(\tau,X_\tau)]=\E[\phi(\ttau,x_0+W_{\ttau})\cE_\infty]$.

\vip

By the It\^o formula,
$$
\beta_{t\land \ttau} U(x_0+W_{t\land \ttau})=\beta_0 U(x_0) + 
\int_0^{t\land \ttau} \beta'_s U(x_0+W_s) \dd s - 2L_t
+\frac12 \int_0^{t\land \ttau} \beta_s \Delta U(x_0+W_s) \dd s,
$$
whence, since $U\geq 0$ and $\beta\geq 0$,
$$
L_\infty \leq \frac12\Big(\beta_0 U(x_0)+\int_0^{\ttau} \beta'_s U(x_0+W_s) \dd s
+\frac12 \int_0^{\ttau} \beta_s \Delta U(x_0+W_s) \dd s  \Big).
$$
Recalling that $\ttau \leq 1$, that $\sup_{[0,\ttau]}|W_s|\leq R\leq 1$
that $x_0 \in \{U \leq A\}$, that $\beta'_s \leq 1/c$ and that
$\sup_{[0,1]} \beta_s \leq \beta_0+1/c$, 
we deduce that
$$
L_\infty \leq \frac 12 \Big(\beta_0+\frac1c\Big)
\Big(A + \sup_{x\in \{U\leq A\}}\sup_{y\in B(x,1)} 
\Big[U(y)+ \frac12 |\Delta U(y)|\Big]\Big)\leq 
C_A(1+\beta_0), 
$$
for some finite constant $C_A>0$ depending on $A$ and $c$. We used that $\cup_{x\in\{U\leq A\}}B(x,1)$
is bounded because $\lim_{|x|\to \infty} U(x)=\infty$.

\vip

Hence $\cE_\infty=
\exp(L_\infty-\frac12\langle L\rangle_\infty)\leq \exp(L_\infty) \leq e^{C_A (1+\beta_0)}$ and
for all measurable $\phi\!:\! \R_+\times
\R^d \to \R_+$, 
\begin{equation}\label{ttb}
\E[\phi(\tau,X_\tau)]\leq  e^{C_A(1+\beta_0)}\E[\phi(\ttau,x_0+W_{\ttau})].
\end{equation}

We now verify that $(\ttau,W_{\ttau})$ has a bounded density, necessarily
supported in $[0,1]\times B(0,1)$. For $r>0$, we introduce $\tau_r=\inf\{t>0 : |W_t|=r\}$.
We have $\ttau=\tau_R\land R$, so that the density of $(\ttau,W_{\ttau})$ is bounded by the
sum of the densities of $(\tau_R,W_{\tau_R})$ and $(R,W_R)$. Recall that $R\sim\cU([1/2,1])$.

\vip

The density of $(R,W_R)$ is $2e^{-|x|^2/(2r)}/(2\pi r)^{d/2}\indiq_{\{r\in [1/2,1],x\in \R^d\}}$,
which is bounded.

\vip

Next, denoting by $\mu_r(s)$ the density of $\tau_r$, we have by scaling that
$\mu_r(s)=r^{-2}\mu_1(r^{-2}s)$, because $\tau_r$ has the same law as $r^2\tau_1$.
One may then check that the density of $(\tau_R,W_{\tau_R})$
is $|x|^{-d-1}\mu_1(|x|^{-2} r) \indiq_{\{r>0,|x|\in [1/2,1]\}}$,
up to some normalization constant. Since $\mu_1$ is bounded, so is the density of
$(\tau_R,W_{\tau_R})$.

\vip

Denoting by $C$ the bound of the density of $(\ttau,W_{\ttau})$, 
we conclude from \eqref{ttb}
that $(\tau,X_{\tau})$ has a density bounded by $Ce^{C_A(1+\beta_0)}\indiq_{\{s\in [0,1],x\in
B(x_0,1)\}}$. The conclusion follows.
\end{proof}

\section{No escape in large time}\label{imp1}

In this section, we prove that $\liminf_{t\to \infty} |X_t|<\infty$.

\begin{prop}\label{rec}
Assume $(A)$ and fix $c>0$, $x_0 \in \R^d$ and $\beta_0 > 0$.
Suppose that there is $\alpha_0>0$ such 
that $\intrd e^{-\alpha_0 U(x)}\dd x <\infty$. For $(X_t)_{t\geq 0}$ the solution to \eqref{eds},
$\liminf_{t\to \infty} |X_t|<\infty$ a.s.
\end{prop}

The crucial point is the following uniform in time {\it a priori} estimate.

\begin{lem}\label{super}
Assume $(A)$, fix $c>0$, $\beta_0 > 0$ and assume that there is 
$\alpha_0\in(0,\beta_0)$ such that $\intrd e^{-\alpha_0 U(x)}\dd x <\infty$.
Let $f_0$ be a probability density on $\R^d$.
Let $(X_t)_{t\geq 0}$ be the solution to \eqref{eds} starting from $X_0$ with law $f_0$.
If $$\kappa(f_0)=\intrd f_0(x) \log (1+f_0(x)e^{\beta_0U(x)})\dd x <\infty,$$
setting $a_0=[\intrd e^{-\alpha_0 U(x)}\dd x]^{-1}$, we informally have
$$
\sup_{t\geq 0} \E[U(X_t)] \leq \frac{\kappa(f_0)-\log(a_0)}{\beta_0-\alpha_0}.
$$
\end{lem}

This relies on a rather indirect entropy computation.
As already mentioned, obtaining a uniform in time moment bound, using the It\^o formula,
would require much more stringent conditions
involving $\nabla U$ and $\Delta U$.
Observe that the computation below is rather original, in that we do not differentiate
the {\it true} relative entropy $\intrd f_t(x)\log(f_t(x)\cZ_{\beta_t} e^{\beta_tU(x)})\dd x$,
where $\cZ_{\beta}=\intrd e^{-\beta U(x)}\dd x$,
but rather the relative entropy {\it without normalization constant}
$\intrd f_t(x)\log(f_t(x)e^{\beta_t U(x)})\dd x$.  Strangely, using the true relative entropy
functional does not seem to provide interesting results.

\begin{proof}
As mentioned in the statement, we give an informal proof. The law $f_t$ of $X_t$
weakly solves
\begin{align}\label{edp}
\partial_t f_t(x)=\frac12 \ddiv [\nabla f_t(x)+\beta_tf_t(x)\nabla U(x)]
=\frac12 \ddiv [e^{-\beta_t U(x)}\nabla(f_t(x)e^{\beta_t U(x)})].
\end{align}
For any smooth $\phi:\R_+ \to \R$, we have, setting $\psi(u)=u\phi'(u)-\phi(u)$
for all $u\geq 0$,
\begin{align}\label{tbru}
&\frac{\dd}{\dd t}\intrd \phi(f_t(x)e^{\beta_tU(x)})e^{-\beta_tU(x)}\dd x\\
=& \intrd \Big[\partial_t f_t(x) \phi'(f_t(x)e^{\beta_tU(x)}) +\beta_t' U(x)f_t(x)
\phi'(f_t(x)e^{\beta_tU(x)}) - \beta_t' 
U(x) \phi(f_t(x)e^{\beta_tU(x)})e^{-\beta_tU(x)} \Big]\dd x\notag\\
=&- \frac12\intrd |\nabla (f_t(x)e^{\beta_tU(x)})|^2\phi''(f_t(x)e^{\beta_tU(x)})e^{-\beta_tU(x)}\dd x
+ \beta'_t \intrd U(x) \psi(f_t(x)e^{\beta_tU(x)})e^{-\beta_tU(x)}\dd x.\notag
\end{align}
For the last equality (first term), we used \eqref{edp} and an integration by parts.

\vip

We now apply \eqref{tbru} with the convex function $\phi(u)=u\log(1+u)$,
for which $\psi(u)=\frac {u^2}{1+u} \leq u$, to find, throwing away the nonpositive term,
$h'_t \leq \beta'_t u_t$, where we have set
$$
h_t=\intrd f_t(x) \log(1+f_t(x)e^{\beta_tU(x)})\dd x \quad\hbox{and}\quad
u_t=\intrd U(x)f_t(x)\dd x=\E[U(X_t)].
$$
But
$$
h_t\geq \intrd f_t(x) \log(f_t(x)e^{\beta_tU(x)})\dd x =
\intrd f_t(x) \log(f_t(x))\dd x + \beta_t u_t
\geq (\beta_t-\alpha_0) u_t + \log(a_0).
$$
We used that $\intrd f(x)\log (f(x)/g(x))\dd x \geq 0$ 
for any pair of probability densities $f$ and $g$ on $\R^d$,
whence $\intrd f_t(x) \log(f_t(x))\dd x \geq \intrd f_t(x) \log(a_0\exp(-\alpha_0 U(x))
\dd x=\log(a_0)-\alpha_0u_t$.

\vip

We conclude, since $h_0=\kappa(f_0)$, that
$$
(\beta_t-\alpha_0) u_t \leq h_t-\log(a_0)\leq \kappa(f_0)-\log(a_0) +\intot \beta'_s
u_s\dd s = \kappa(f_0)-\log(a_0) + \intot \frac{\beta'_s}{\beta_s-\alpha_0} 
(\beta_s-\alpha_0)u_s\dd s,
$$
whence, by the Gronwall lemma, 
$$
(\beta_t-\alpha_0) u_t \leq [\kappa(f_0)-\log(a_0)]\exp\Big(\intot\frac{\beta'_s \dd s}
{\beta_s-\alpha_0}    \Big)
= [\kappa(f_0)-\log(a_0)] \frac{\beta_t-\alpha_0}{\beta_0-\alpha_0}.
$$
Consequently, $\E[U(X_t)]=u_t\leq [\kappa(f_0)-\log(a_0)]/[\beta_0-\alpha_0]$ for all $t\geq 0$.
\end{proof}

We now try to deduce from this informal computation the rigorous results we need.

\begin{lem}\label{r1}
If $\nabla U$ is bounded together with all its derivatives
and if the initial density $f_0$ belongs to $C_c(\R^d)$,
the \emph{a priori} estimate of Lemma \ref{super} rigorously holds true for the
solution \eqref{eds} starting from $X_0\sim f_0$.
\end{lem}

\begin{proof}
We first justify rigorously \eqref{tbru}, for all $t\in (0,\infty)$, with $\phi(u)=u \log(1+u)$.
Recall that $f_t$ is the law of $X_t$.
Since $U$ has at most linear growth and $\nabla U$ is bounded, it is (widely) 
enough to check that
$(f_t(x))_{t>0,x\in\R^d}$ is a strong solution to \eqref{edp}, i.e. 
$2\partial_t f_t(x)=\Delta f_t(x)+
\beta_t\nabla U(x)\cdot\nabla f_t(x)+\beta_t f_t(x)\Delta U(x)$
on $(0,\infty)\times\R^d$ and satisfies,
for all $0<t_0<t_1$, for some constants
$C_{t_0,t_1}>0$ and $\lambda_{t_0,t_1}>0$,
$$ 
\forall t\in [t_0,t_1],\; x\in\R^d, \quad 
f_t(x) + |\partial_t f_t(x)|+|\nabla f_t(x)|+ |D^2 f_t(x)| \leq C_{t_0,t_1} \exp(-\lambda_{t_0,t_1} |x|^2).
$$
To prove those bounds, we use classical results found in Friedman \cite{f}, that apply
to uniformly parabolic equations with bounded and Lipschitz coefficients (actually, H\"older is enough):
by \cite[Chapter 1, Theorem 12]{f}, we have 
$f_t(x)=\intrd \Gamma(x,t;\xi,0)f_0(\xi)\dd \xi$, with,
for some $C_T>0$ and $\lambda>0$, for all $t\in [0,T]$, all $x\in \R^d$,
\begin{equation}\label{rreeuu}
|\Gamma(x,t;\xi,0)| + t^{1/2} |\nabla_x\Gamma(x,t;\xi,0)| 
+ t |D^2 \Gamma(x,t;\xi,0)| 
\leq C_T t^{-d/2} e^{-\lambda |x-\xi|^2/t}.
\end{equation}
The above estimates for $\Gamma$ and $\nabla \Gamma$ are nothing but 
\cite[Chapter 1, Equations (6.12) and (6.13)]{f},
and the estimate on $D^2 \Gamma$ is proved similarly, using \cite[Chapter 1, Equation (4.11)]{f}.
The Gaussian upper-bounds of $f_t(x)$, $|\nabla f_t(x)|$ and $|D^2 f_t(x)|$ follow, 
because $f_0 \in C_c(\R^d)$.
Finally, the bound on $\partial_t f_t(x)$ follows from the fact that
$2|\partial_tf_t(x)|\leq |\Delta f_t(x)|+\beta_t||\nabla U||_\infty|\nabla f_t(x)|
+\beta_t||\Delta U||_\infty f_t(x)$.

\vip

Hence, all the arguments in the proof of Lemma \ref{super} are correct for $t\in (0,\infty)$, 
and we conclude that
for all $t_0>0$, $\sup_{t\geq t_0}\E[U(X_t)]\leq (\beta_{t_0}-\alpha_0)^{-1}(\kappa(f_{t_0})-\log(a_0))$.
To complete the proof, the only issue is to show that 
$\lim_{t_0\to 0+}\kappa(f_{t_0})=\kappa(f_0)$. This can be deduced from the continuity of
$f_t(x)$ on $[0,\infty)\times \R^d$, see \cite[Chapter 1, Section 7]{f},
and the fact that there are $C_T>0$ and $\lambda_T>0$ such that
$f_t(x) \leq C_T e^{-\lambda_T |x|^2}$ for all $t\in [0,T]$ and $x\in \R^d$.
This follows from \eqref{rreeuu}, the fact that
$f_t(x)=\intrd \Gamma(x,t;\xi,0)f_0(\xi)\dd \xi$ and that $f_0 \in C_c(\R^d)$.
\end{proof}

We can now prove the main result of this section.

\begin{proof}[Proof of Proposition \ref{rec}] We assume $(A)$ and that $\intrd e^{-\alpha_0 U(x)}\dd x <\infty$
for some $\alpha_0>0$. We fix $c>0$, $x_0 \in \R^d$ and $\beta_0 > 0$ and aim to check that
for $(X_t)_{t\geq 0}$ the solution to \eqref{eds},
$\liminf_{t\to \infty} |X_t|<\infty$ a.s. We divide the proof in four steps.
\vip

{\it Step 1.} We of course may assume additionally that $\beta_{0}>\alpha_0$:
fix $t_0\geq 0$ large enough so that $\beta_{t_0}>\alpha_0$ and observe that
$(X_{t_0+t})_{t\geq 0}$ solves \eqref{eds}, with $x_0$ replaced by $X_{t_0}$
and $\beta_0$ replaced by $\beta_{t_0}$ (and with the Brownian motion
$(B_{t_0+t}-B_{t_0})_{t\geq 0}$). Since $\liminf_{t\to \infty} |X_t|=\liminf_{t\to \infty} |X_{t_0+t}|$,
the conclusion follows.

\vip

{\it Step 2.}  From now on, we assume that $\beta_0>\alpha_0$.
We introduce the stopping time $\tau\in [0,1]$
as in Lemma \ref{tau}. We recall that $\sup_{[0,\tau]}|X_t-x_0|\leq 1$ and that for
$h \in L^1([0,1]\times B(x_0,1))$ the density of $(\tau,X_\tau)$, 
there is $C>0$ (depending on $x_0$ and $\beta_0$) such that 
$h(u,x)\leq C \indiq_{\{u \in [0,1],x\in B(x_0,1)\}}$.

\vip

{\it Step 3.} For $n\geq |x_0|+1$, we introduce $U_n \in C^\infty(\R^d)$ such that
$U_n(x)=U(x)$ for all $x \in B(0,n)$ and $U_n(x)=|x|$ as soon as $|x|\geq n+1$,
with furthermore $U_n(x)\geq \min(U(x),|x|)-1$ for all $x\in \R^d$. Then $\nabla U_n$
is bounded together all its derivatives. We denote by $(X^n_t)_{t\geq 0}$
the solution to \eqref{eds}, with $U_n$ instead of $U$.
By a classical uniqueness argument (using that $\nabla U$ is locally 
Lipschitz continuous),
$X$ and $X^n$ coincide until they reach $B(0,n)^c$. 
In particular, $X_t=X^n_t$ for all $t\in [0,\tau]$ and, 
setting 
$$
\zeta_n=\inf\{t\geq 0 : |X_{\tau+t}|\geq n\}=\inf\{t\geq 0 : |X_{\tau+t}^n|\geq n\},
$$
it a.s. holds that $X^n_{\tau+t}=X_{\tau+t}$ for all $t\in [0,\zeta_n]$.
Since $\lim_n \zeta_n=\infty$ a.s., we conclude that for all $t\geq 0$,
$\lim_n U_n(X^n_{\tau+t})=U(X_{\tau+t})$ a.s. 

\vip

As we will check in Step 4, 
\begin{equation}\label{oob}
\sup_{n\geq |x_0|+ 1} \sup_{t\geq 0}\E[U_n(X_{\tau+t}^n)]<\infty.
\end{equation}
By the Fatou lemma, we will conclude that  $\sup_{t\geq 0}\E[U(X_{\tau+t})]<\infty$. By the Fatou Lemma again,
this will imply that $\E[\liminf_{t\to \infty}U(X_t)]<\infty$.
Since $\lim_{|x|\to \infty} U(x)
=\infty$ by $(A)$, 
this will show that $\liminf_{t\to \infty} |X_t|<\infty$ a.s. and thus complete the proof.

\vip

{\it Step 4.} Here we verify \eqref{oob}.
%
Denote, for $x\in \R^d$ and $\beta>0$, by $f^{n,x,\beta}_t$ 
the law at time $t$ of the solution to \eqref{eds}
with $x_0=x$, with $\beta_0$ replaced by $\beta$ and with $U_n$ instead of $U$. 
We then have, since $h$ is the density of $(\tau,X_\tau)=(\tau,X^n_\tau)$,
$$
\E[U_n(X^n_{\tau+t})]=\E[\E[U_n(X^n_{\tau+t})\vert \cF_\tau]]= \int_{[0,1]\times B(x_0,1)} h(u,x) 
\Big[\intrd U_n(y)f^{n,x,\beta_u}_t(\dd y)\Big] \dd u \dd x.
$$
Consider any probability density $f_0\in C_c(\R^d)$ such that 
$f_0>c \indiq_{B(x_0,1)}$, for some constant $c>0$. We thus have 
$h(u,x)\leq C\indiq_{\{u \in [0,1],x\in B(x_0,1)\}}\leq (C/c) f_0(x)$, and write 
$$
\E[U_n(X^n_{\tau+t})]\leq  \frac Cc \sup_{u\in[0,1]}\intrd f_0(x) 
\Big[\intrd U_n(y)f^{n,x,\beta_u}_t(\dd y)\Big]\dd x
= \frac Cc  \sup_{u\in[0,1]}\E[U_n(Y^{n,u}_{t})],
$$
where $(Y^{n,u}_t)_{t\geq 0}$ is the solution to \eqref{eds} starting from $X_0\sim f_0$, with
$\beta_0$ replaced by $\beta_u$ and $U$ by $U_n$. To conclude the step, it only remains to
verify that $\sup_{n\geq |x_0|+1} \sup_{u\in[0,1]}\sup_{t\geq 0} \E[U_n(Y^{n,u}_{t})] <\infty$.

\vip

But Lemmas \ref{super} and \ref{r1} tell us that, setting 
$\kappa_{n,u}(f_0)=\intrd f_0(x)\log(1+f_0(x)e^{\beta_u U_n(x)})\dd x$ and $a_n=[\intrd e^{-\alpha_0 U_n(x)}\dd x]^{-1}$, 
it holds that 
$$
\E[U_n(Y^{n,u}_{t})]\leq \frac{\kappa_{n,u}(f_0)-\log a_n}{\beta_u-\alpha_0}.
$$
This last quantity is uniformly bounded, because 

\vip

\noindent $\bullet$ $\beta_u\geq \beta_0>\alpha_0$ for all $u\in[0,1]$ (by Step 1); 

\vip

\noindent $\bullet$ $\sup_{n\geq |x_0|+1,u\in [0,1]} \kappa_{n,u}(f_0) < \infty$,
since $f_0 \in C_c(\R^d)$, $\beta_u \leq \beta_1$ for all $u\in [0,1]$ and $U_n(x)=U(x)$ for all $x\in$ Supp $f_0$ 
if $n$ is large enough;

\vip
\noindent  $\bullet$ $\sup_{n\geq |x_0|+1} (-\log a_n) < \infty$, since
$\intrd e^{-\alpha_0 U_n(x)}\dd x\leq e^{\alpha_0}[\intrd e^{-\alpha_0 U(x)}\dd x+\intrd e^{-\alpha_0 |x|}\dd x] <\infty$,
recall that $U_n(x)\geq \min\{U(x),|x|\}-1$.
\end{proof}

\section{Localization and absorption}\label{imp2}

Here we prove that on the event where $\sup_{t\geq 0} |X_t|<\infty$,
the simulated annealing procedure is successful.
We also check that each time the process $(X_t)_{t\geq 0}$ comes back in a given compact,
it has a large probability to be absorbed forever in a (larger) compact.

\begin{lem}\label{fmt}
Assume $(A)$, fix $c>c_*$, $x_0\in \R^d$ and $\beta_0>0$ and consider the solution
$(X_t)_{t\geq 0}$ to \eqref{eds}. For any $\e>0$,
$$
\lim_{t\to \infty}\PR\Big(\sup_{s\geq 0} |X_s|<\infty \hbox{ and } U(X_t)>\e\Big)=0.
$$
\end{lem}

\begin{prop}\label{crucial}
Assume $(A)$ and fix $c>c_*$.
For any $A\geq 1$, there is $b_A>1$ and $K_A>A$ such that if 
$x_0\in \{U\leq A\}$ and $\beta_0\geq b_A$,
for $(X_t)_{t\geq 0}$ the solution to \eqref{eds}, we have
$$
\PR\Big(\sup_{t\geq 0} U(X_t) \leq K_A\Big)\geq \frac12.
$$
\end{prop}

The rest of the section is dedicated to the proof of these two results.
Lemma \ref{fmt} will easily follow from a result of Holley-Kusuoka-Stroock \cite{hks}
concerning the compact case.

\vip

Concerning Proposition \ref{crucial}, let us recall from 
Holley-Kusuoka-Stroock \cite{hks}, see also Miclo \cite{mic2}, that
in the compact setting, $\limsup_{t\to \infty} U(X_t)=c$ a.s. and moreover for any
$\e>0$, if $x_0$ belongs to a connected component of $\{U\leq c+\e\}$ containing
a global minimum of $U$, it holds that $\PR(\sup_{t\geq 0} U(X_t) \leq c+\e)>0$.
This immediately extends to the non-compact setting, since the set $\{U\leq c+\e\}$
is compact. Unfortunately, such a result is not uniform in $\beta_0>0$, and we really need
a uniform bound, see Step 2 of the proof of Proposition \ref{mm} in Section \ref{con}.
We believe it is not possible to deduce Proposition \ref{crucial} from \cite{hks,mic2}.
At this end, we have to work hard, following the ideas of \cite{hks}, taking much less care about 
many constants and obtaining
much less precise results (e.g. it might be possible to control $K_A$ in Proposition \ref{crucial})
but carefuly tracking the dependence in $\beta_0$ and $x_0$.

\vip

In the whole section, we assume $(A)$ and work with some fixed $c>c_*$. 
We introduce some notation.

\begin{nota}\label{tore}
Let $K\geq 1$.

\vip

(a) We consider $L_K>0$ such that $\{U\leq K\}\subset [-(L_K-1),(L_K-1)]^d$.
We denote by $M_K$ the torus $[-L_k,L_K)^d$, that is $\R^d$ 
quotiented by the equivalence relation
$x\sim y$ if and only if for all $i=1,\dots,d$, $(x_i-y_i)/(2L_K) \in \Z$.

\vip

(b) We also consider $U_K \in C^\infty(M_K)$ such that $\min_{M_K}U_K=0$, such that
$U_K(x)=U(x)$ for all $x\in \{U\leq K\}$, and such that
$$
c_*^K=\sup\{E_K(x,y) : x,y\in M_K\}\leq c_*,
$$
where $E_K(x,y)=\inf \{\max_{t\in [0,1]}U_K(\gamma_t)-U_K(x)-U_K(y) \;:\;
\gamma \in C([0,1],M_K), \gamma_0=x,\gamma_1=y\}$.

\vip

(c) For $x_0 \in \{U\leq K\}\subset M_K$ and $\beta_0>0$, we introduce
the inhomogeneous $M_K$-valued diffusion
\begin{equation}\label{edsK}
X^K_t = x_0+ B_t - \frac12 \intot \beta_s \nabla U_K(X^K_s)\dd s  \hbox{ modulo } 2L_K,
\end{equation}
where $(B_t)_{t\geq 0}$ is a $d$-dimensional Brownian motion, where
$\beta_t=c^{-1}\log(e^{c\beta_0}+t)$ as in \eqref{eds} and
$$
\hbox{for $x=(x_1,\dots,x_d) \in \R^d$, } \quad
x \hbox{ modulo } 2L_K=
\Big(x_i - 2L_K\Big\lfloor \frac{x_i+L_K}{2L_K}\Big\rfloor \Big)_{i=1,\dots,d} 
\in [-L_K,L_K)^d.
$$
\end{nota}

For point (b), it suffices to choose a smooth version of $U_K=\min\{U,K\}$, see 
\cite[Step 1 of the proof of Lemma 6]{fmt}.
Since $U_K=U$ on $\{U\leq K\}$ and since $U$ is locally Lipschitz continuous, 
a simple uniqueness argument shows the following.

\begin{rk}\label{compegpascomp}
For any $K\geq 1$, any $x_0\in \{U\leq K\}$, any $\beta_0>0$,
for $(X_t)_{t\geq 0}$ the solution to \eqref{eds} and $(X^K_t)_{t\geq 0}$
the solution to \eqref{edsK}, both driven by the same Brownian motion,
it holds that
\begin{align*}
\Big\{\sup_{t\geq 0} U_K(X^K_t)\leq K\Big\}= 
\Big\{\sup_{t\geq 0} U_K(X^K_t)\leq K,\sup_{t\geq 0} |X^K_t-X_t|=0\Big\}
= \Big\{\sup_{t\geq 0} U(X_t)\leq K\Big\}.
\end{align*}
\end{rk}

We can now give the

\begin{proof}[Proof of Lemma \ref{fmt}]
By \cite[Theorem 2.7]{hks} and since $c>c_*\geq c_*^K$, $U_K(X^K_t)\to 0$ in probability, as $t\to \infty$,
for each $K\geq 1$. We fix $\eta>0$. Since $\lim_{|x|\to\infty} U(x)=\infty$, there is $K_\eta>0$
such that $\PR(\sup_{s\geq 0} |X_s|<\infty, \sup_{s\geq 0} U(X_s)>K_\eta)\leq \eta$.
We then write, using Remark \ref{compegpascomp},
\begin{align*}
\PR\Big(\sup_{s\geq 0} |X_s|<\infty \hbox{ and } U(X_t)>\e\Big)\leq &
\eta + \PR\Big(\sup_{s\geq 0} U(X_s)\leq K_\eta \hbox{ and } U(X_t)>\e\Big)\\
=& \eta + \PR\Big(\sup_{s\geq 0} U_{K_\eta}(X_s^{K_\eta})\leq K_\eta \hbox{ and } U_{K_\eta}(X_t^{K_\eta})>\e\Big)\\
\leq & \eta + \PR\Big(U_{K_\eta}(X_t^{K_\eta})>\e\Big).
\end{align*}
We conclude that $\limsup_{t\to\infty}\PR(\sup_{s\geq 0} |X_s|<\infty 
\hbox{ and } U(X_t)>\e)\leq \eta$,
whence the result since $\eta>0$ is arbitrarily small.
\end{proof}

We next introduce the invariant probability measure of the time-homogeneous version of \eqref{edsK}.

\begin{rk}\label{ttt}
There is a constant $\kappa_0>0$ such that, for all $K\geq 1$, all $\beta>0$,
it holds that
$$
\cZ^K_\beta:= \intMK \exp(-\beta U_K(x))\dd x \geq \kappa_0(\beta+1)^{-d}.
$$
We also have $\cZ^K_\beta\leq (2L_K)^d$.
We introduce the probability density 
$$
\mu_\beta^K(x)=(\cZ^K_\beta)^{-1} \exp(-\beta U_K(x)), \quad x\in M_K.
$$
\end{rk}

\begin{proof}
Since $\min_{\R^d} U=0$, there is $x_* \in \R^d$ such that $U(x_*)=0$. Fix $r_*>0$
such that $B(x_*,r_*)\subset \{U\leq 1\}\subset M_K$. 
Denote by $C=\sup_{B(x_*,r_*)} |\nabla U|$.
For all $K\geq 1$, all $x\in B(x_*,r_*)$, we have that $x\in \{U\leq 1\}\subset M_K$ and
$U_K(x)=U(x)\leq C|x-x_*|$. Hence for all $\beta>0$,
$$
\cZ_\beta^K\geq \int_{B(x_*,r_*)} \exp(-\beta C|x-x_*|)\dd x
\geq e^{-1}\hbox{Vol}(B(x_*,r_*\land (1/(C\beta))),
$$
from which the lower-bound follows. The upper-bound is trivial.
\end{proof}

As a final preliminary, we recall the crucial
spectral gap estimate of Holley-Kusuoka-Stroock \cite[Theorem 1.14
and Remark 1.16]{hks}, in the special case of the torus. We use that
$c_*^K\leq c_*$, see Notation \ref{tore}-(b) (in the notation of \cite{hks},
$m=c_*^K$).

\begin{lem}[Holley-Kusuoka-Stroock] \label{hksl}
Fix $K\geq 1$. There is a constant $\gamma_K>0$ such that for all 
$\phi \in C^1(M_K)$, for all $\beta>0$,
$$
\intMK |\nabla \phi(x)|^2 \mu^K_\beta(x)\dd x \geq \lambda_K(\beta)
\intMK\Big(\phi(x)-\intMK\phi(y)\mu^K_\beta(y)\dd y \Big)^2 \mu^K_\beta(x)\dd x,
$$
with
$$
\lambda_K(\beta)=\gamma_K (\beta+1)^{2-5d}\exp(-\beta c_*).
$$
\end{lem}

The constant $\gamma_K$ drastically depends on $K$ but, as we will see, 
this is not an issue.

\begin{lem}\label{hksb}
Fix $K\geq 1$. There is a constant $b_K^{(1)}>0$ 
such that if $\beta_0\geq b_K^{(1)}$, then 
for any density $f_0^K \in C(M_K)$, for $f^K_t$ the density of $X^K_t$, 
the solution to \eqref{edsK} starting from $X_0^K\sim f_0^K$,
$$
\forall \; t\geq 0,\quad \intMK \frac{(f^K_t(x))^2}{\mu^K_{\beta_t}(x)} \dd x \leq
\max \Big(2,\intMK \frac{(f^K_0(x))^2}{\mu^K_{\beta_0}(x)} \dd x\Big).
$$
\end{lem}

\begin{proof}
The function $(f^K_t(x))_{t\geq 0,x\in M_K}$ is a weak solution to the uniformly parabolic equation
$\partial_t f^K_t(x) = \frac12\ddiv(\nabla f^K_t(x)+\beta_tf^K_t(x)\nabla U_K(x))$. 
It can be seen as a periodic solution of the same equation in $\R^d$, with $U_K$ and $f_0^K$
replaced by their periodic continuation. We thus can apply some classical results, see
Friedman \cite[Chapter 1, Theorems 10 and 12]{f} and conclude that $(f^K_t(x))_{t\geq 0,x\in M_K}$ belongs to
$C([0,\infty)\times M_K) \cap C^{1,2}((0,\infty)\times M_K)$.
The periodic continuation of $f_0^K$ has an infinite mass, but this is allowed by \cite{f}.
Since furthermore $M_K$ is bounded, all the computations below are easily justified.

\vip

We introduce
$$
\varphi(t)=\intMK \frac{(f^K_t(x))^2}{\mu^K_{\beta_t}(x)} \dd x 
= \cZ^K_{\beta_t}\intMK(f^K_t(x))^2e^{\beta_t U_K(x)}\dd x.
$$
Since $(\cZ^K_{\beta_t})'=-\beta_t'\intMK U_K(x)e^{-\beta_t U_K(x)}\dd x\leq0$, we have, for all $t>0$,
$$
\varphi'(t)\leq  \cZ^K_{\beta_t} \intMK 2[\partial_t f_t^K(x)] f_t^K(x)e^{\beta_t U_K(x)}   \dd x 
+ \beta_t' \cZ^K_{\beta_t} \intMK U_K(x)(f^K_t(x))^2e^{\beta_t U_K(x)}\dd x.
$$
Recalling that $\partial_t f^K_t(x)= \frac12\ddiv(\nabla f^K_t(x)+\beta_tf_t^K(x)
\nabla U_K(x))$, proceeding 
to an integration by parts in the first term and to a rough upper-bound 
in the second one, we find
\begin{align*}
\varphi'(t) \leq& -\cZ^K_{\beta_t} \intMK 
|\nabla f^K_t(x)+\beta_t f_t^K(x)\nabla U_K(x)|^2e^{\beta_t U_K(x)}\dd x
+ \beta'_t ||U_K||_\infty \varphi(t)\\
=&-\intMK \Big|\nabla\Big(\frac{f^K_t(x)}{\mu^K_{\beta_t}(x)}\Big)\Big|^2\mu^K_{\beta_t}(x)\dd x 
+ \beta'_t ||U_K||_\infty \varphi(t) .
\end{align*}
By Lemma \ref{hksl} with $\phi(x)=f^K_t(x)/\mu^K_{\beta_t}(x)$, for which $\intMK \phi(y)\mu^K_{\beta_t}(y)\dd y=1$ , 
we conclude that 
$$
\varphi'(t) \leq -\lambda_K(\beta_t) \intMK 
\Big[\frac{f^K_t(x)}{\mu^K_{\beta_t}(x)}-1\Big]^2\mu^K_{\beta_t}(x)\dd x 
+\beta'_t ||U_K||_\infty \varphi(t)=-\lambda_K(\beta_t)[\varphi(t)-1]
+\beta'_t ||U_K||_\infty \varphi(t).
$$
But we know from Lemma \ref{hksl} that for all $\beta>0$,  
$$
\lambda_K(\beta)\geq\gamma_K (\beta+1)^{2-5d}e^{-\beta c_*}\geq \gamma_K'e^{-\beta(c+c_*)/2}
$$ 
for some other constant $\gamma_K'>0$, since $c>c_*$. Setting 
$\alpha=(c-c_*)/(2c) \in (0,1)$, so that
$(c+c*)/2=c(1-\alpha)$, recalling that $\beta_t=\log(e^{c\beta_0}+t)/c$, we conclude that
$$
\varphi'(t) \leq -\frac{\gamma_K'}{(e^{c\beta_0}+t)^{1-\alpha}}
[\varphi(t)-1]+\frac{ ||U_K||_\infty}
{c(e^{c\beta_0}+t)}\varphi(t).
$$
Let $b_K^{(1)}=\frac1{\alpha c} \log(2||U_K||_\infty/(\gamma'_K c))$, 
so that if $\beta_0\geq b_K^{(1)}$, for all $t\geq 0$, 
$$
\frac{||U_K||_\infty}{c(e^{c\beta_0}+t)}\leq \frac{\gamma_K'}{2(e^{c\beta_0}+t)^{1-\alpha}}, 
$$
whence
$$
\varphi'(t) \leq -\frac{\gamma_K'}{2(e^{c\beta_0}+t)^{1-\alpha}}[\varphi(t)-2].
$$
We classically conclude that indeed, if $\beta_0\geq b_K^{(1)}$,
then for all $t\geq 0$, $\varphi(t) \leq \max(2,\varphi(0))$.
\end{proof}

From the previous lemma and the Cauchy-Schwarz inequality, we deduce the following.

\begin{lem}\label{cs}
For $A\geq 1$,
let $D_A=2C^{(1)}_A+C^{(2)}_A+1+4c$ and $K_A=D_A+1$, 
where $C^{(1)}_A$ was introduced in Lemma \ref{tau}
and where $C^{(2)}_A=\sup_{x\in \{U\leq A\}}\sup_{y\in B(x,2)} U(y)$. 
There is a constant $C^{(3)}_A>0$
such that, if $\beta_0\geq b^{(1)}_{K_A}$ (see Lemma \ref{hksb}) 
and $x_0\in \{U\leq A\}$, it holds that $\sup_{[0,\tau]} U_{K_A}(X^{K_A}_t)\leq D_A$ a.s. and
$$
\forall\; t\geq 0,\quad \PR(U_{K_A}(X^{K_A}_{\tau+t}) \geq D_A)
\leq \frac{C^{(3)}_A}{(e^{c\beta_0}+t)^2},
$$
where $(X^{K_A}_t)_{t\geq 0}$ is the solution to \eqref{edsK} starting from $x_0$ and where
$\tau$ is the stopping time introduced in Lemma \ref{tau} 
(for the solution $(X_t)_{t\geq 0}$
to \eqref{eds} driven by the same Brownian motion as $(X^{K_A}_t)_{t\geq 0}$).
\end{lem}

\begin{proof}
We fix $A\geq 1$, $\beta_0\geq b^{(1)}_{K_A}$ and $x_0 \in \{U\leq A\}$.
First, since 
$B(x_0,1)\subset \{U\leq D_A\}$
(because $D_A\geq C^{(2)}_A$), since $\sup_{[0,\tau]}|X_t-x_0|\leq 1$ and since
$D_A \leq K_A$, 
Remark \ref{compegpascomp} tells us that 
$X_t=X^{K_A}_t$ for all $t\in [0,\tau]$. In particular, 
$\sup_{[0,\tau]} U_{K_A}(X^{K_A}_t)\leq D_A$ and the law of $(\tau,X^{K_A}_\tau)=(\tau,X_\tau)$
has a density $h(u,x)$ bounded by $e^{C^{(1)}_A(\beta_0+1)}\indiq_{\{u\in[0,1],x\in B(x_0,1)\}}$, 
see Lemma \ref{tau}.

\vip

Denote, for $x\in \R^d$ and $\beta>0$, by $f^{K_A,x,\beta}_t$
the law of the solution of \eqref{edsK} with $K=K_A$, with $x_0=x$ and 
with $\beta_0$ replaced by $\beta$.
We then have
\begin{align*}
\PR(U_{K_A}(X^{K_A}_{\tau+t}) \geq D_A) = &
\E[\PR(U_{K_A}(X^{K_A}_{\tau+t}) \geq D_A |\cF_\tau)] \\
=& \int_{[0,1]\times B(x_0,1)} h(u,x)
\Big[\int_{\{U_{K_A}\geq D_A\}} f^{K_A,x,\beta_u}_t(\dd y)\Big] \dd u \dd x.
\end{align*}
Consider a probability density $f_0\in C_c(\R^d)$ such that 
$(2v_d)^{-1}\indiq_{B(x_0,1)}\leq f_0 \leq \indiq_{B(x_0,2)}$, $v_d$ 
being the volume of the unit ball.
We write 
\begin{align}\label{a0}
\PR(U_{K_A}(X^{K_A}_{\tau+t})\geq D_A)
\leq&  2v_de^{C^{(1)}_A(\beta_0+1)} \sup_{u\in[0,1]}\intrd \Big[\int_{\{U_{K_A}\geq D_A\}} 
f^{K_A,x,\beta_u}_t(\dd y)\Big] f_0(x)\dd x\\
=& 2v_de^{C^{(1)}_A(\beta_0+1)} \sup_{u\in[0,1]}  \PR(U_{K_A}(Y^{A,u}_{t}) \geq D_A),\notag
\end{align}
where $(Y^{A,u}_t)_{t\geq 0}$ is the solution to \eqref{edsK} with $K=K_A$, 
starting from $Y_0\sim f_0$, with
$\beta_0$ replaced by $\beta_u$. We now denote by $f^{A,u}_t$ the density 
of $Y^{A,u}_t$ and use the Cauchy-Schwarz inequality to write
\begin{align}\label{a1}
\PR(U_{K_A}(Y^{A,u}_{t}) \geq D_A)=& \int_{\{U_{K_A}\geq D_A\}} f^{A,u}_t(x)\dd x \\
\leq&
\Big(\int_{\{U_{K_A}\geq D_A\}} \mu^{K_A}_{\beta_{t+u}}(x)\dd x\Big)^{1/2}
\Big(\int_{M_{K_A}} \frac{[f^{A,u}_t(x)]^2\dd x}
{\mu^{K_A}_{\beta_{t+u}}(x)}\Big)^{1/2}.\notag
\end{align}
By Lemma \ref{hksb}, we know that, since $\beta_u\geq \beta_0\geq b_{K_A}^{(1)}$,
$$
\int_{M_{K_A}} \frac{[f^{A,u}_t(x)]^2\dd x}{\mu^{K_A}_{\beta_{t+u}}(x)}\leq 2\lor
\int_{M_{K_A}} \frac{[f_0(x)]^2\dd x}{\mu^{K_A}_{\beta_{u}}(x)} \leq 2 \lor \int_{B(x_0,2)}
\frac{\dd x}{\mu^{K_A}_{\beta_{u}}(x)} = 2 \lor \int_{B(x_0,2)}\cZ_{\beta_u}^{K_A}e^{\beta_u U(x)}\dd x.
$$
Recalling the definition of $C^{(2)}_A$ and that 
$\cZ^{K_A}_{\beta_{u}}\leq (2L_{K_A})^d$, see Remark \ref{ttt}, we find
\begin{equation}\label{a2}
\int_{M_{K_A}} \frac{[f^{A,u}_t(x)]^2\dd x}{\mu^{K_A}_{\beta_{t+u}}(x)}\leq 2\lor[2^dv_d (2L_{K_A})^de^{C^{(2)}_A \beta_u}]
=2^dv_d (2L_{K_A})^de^{C^{(2)}_A \beta_u} .
\end{equation}
Next, since the volume of $M_{K_A}$ is smaller than $(2L_{K_A})^d$, we have
$$
\int_{\{U_{K_A}\geq D_A\}} \mu^{K_A}_{\beta_{t+u}}(x)\dd x
\leq \frac{(2L_{K_A})^d e^{-\beta_{t+u}D_A}}{\cZ^{K_A}_{\beta_{t+u}}}.
$$
By Remark \ref{ttt} again,
\begin{equation}\label{a3}
\int_{\{U_{K_A}\geq D_A\}} \mu^{K_A}_{\beta_{t+u}}(x)\dd x
\leq \frac{(2L_{K_A})^d (\beta_{t+u}+1)^d e^{-\beta_{t+u}D_A}}{\kappa_0}\leq C_A e^{-\beta_{t+u} (D_A-1)},
\end{equation}
for some constant $C_A>0$ of which we now allow the value to change from line to line.
Gathering \eqref{a0}-\eqref{a1}-\eqref{a2}-\eqref{a3}, we find
$$
\PR(U_{K_A}(X^{K_A}_{\tau+t})\geq D_A) 
\leq C_A e^{C^{(1)}_A (\beta_0+1)} \sup_{u\in[0,1]} e^{-\beta_{t+u} (D_A-1)/2 + C_A^{(2)} \beta_u/2}.
$$
Since $s\mapsto \beta_s$ is non-decreasing and $D_A-1>C_A^{(2)}$,
$$
\PR(U_{K_A}(X^{K_A}_{\tau+t})\geq D_A) 
\leq C_A e^{C^{(1)}_A \beta_t} \sup_{u\in[0,1]} e^{-\beta_{t+u} (D_A-1-C_A^{(2)})/2}\leq
C_A e^{-\beta_{t}(D_A-1-C_A^{(2)}-2C_A^{(1)})/2}.
$$
Recalling finally that $D_A=2C^{(1)}_A+C^{(2)}_A+1+4c$
and that $\beta_t=c^{-1}\log(e^{c\beta_0}+t)$, we conclude that
$$
\PR(U_{K_A}(X^{K_A}_{\tau+t})\geq D_A) \leq C_A e^{-2c\beta_{t}} = \frac{C_A}{(e^{c\beta_0}+t)^2}
$$
as desired.
\end{proof}

We finally give the

\begin{proof}[Proof of Proposition \ref{crucial}.]
We fix $A\geq 1$ and introduce $D_A$, $K_A=D_A+1$, $b^{(1)}_{K_A}$ 
and $C^{(3)}_A$ as in Lemma \ref{cs}.
We will show that one can find $b_A>b^{(1)}_{K_A}$ such that if $\beta_0>b_A$ and 
$x_0\in \{U\leq A\}$,
the solution $(X^{K_A}_t)_{t\geq 0}$ to \eqref{edsK} satisfies
$$
\PR\Big(\sup_{t\geq 0} U_{K_A}(X_t^{K_A}) \leq K_A\Big)\geq \frac12.
$$
By Remark \ref{compegpascomp}, this will show the result.
By Lemma \ref{cs}, we have $\sup_{[0,\tau]} U_{K_A}(X_t^{K_A})\leq D_A\leq K_A$ a.s.,
so that we only have to check that
$$
\PR\Big(\sup_{t\geq 0} U_{K_A}(X_{\tau+t}^{K_A}) \leq K_A\Big)\geq \frac12.
$$

\vip

We consider $\phi_A \in C^\infty(\R_+)$, with values in $[0,1]$, 
such that $\phi_A=0$ outside $[D_A,K_A]$ and such that $\phi_A((D_A+K_A)/2)=1$,
and we introduce $\psi_A=\phi_A \circ U_{K_A}: M_{K_A}\to [0,1]$. 
Setting 
$$
\cL_{\beta_t}^A\psi_A(x)=\frac12(\Delta \psi_A(x)-\beta_t\nabla \psi_A(x)\cdot\nabla U_{K_A}(x)),
$$
we have $|\cL_{\beta_t}^A\psi_A(x)|\leq C_A^{(4)} (1+\beta_t) \indiq_{\{U_{K_A}(x)\geq D_A\}}$, where
$C_A^{(4)}$ is a constant involving the supremum on $M_{K_A}$ 
of $U_{K_A}$ and its two first derivatives.

\vip

We now fix $b_A>b^{(1)}_{K_A}$ such that for all $\beta_0\geq b_A$,
$$
\int_0^\infty \frac{C^{(3)}_A C^{(4)}_A [1+\log(e^{c \beta_0}+t+1)]}
{c (e^{c \beta_0} + t)^2} \dd t
=\int_{e^{c\beta_0}}^\infty \frac{C^{(3)}_A C^{(4)}_A [1+\log(s+1)]}
{c  s^2} \dd s
\leq \frac 1{40}.
$$

By It\^o's formula and since $\psi_A(X_{\tau}^{K_A})=0$ 
(because $U_{K_A}(X_{\tau}^{K_A})\leq D_A$),
$$
\psi_A(X_{\tau+t}^{K_A})=M_t+R_t, 
$$
where
$(M_t)_{t\geq 0}$ is a martingale issued from $0$ 
and where 
$$R_t=\intot \cL_{\beta_{\tau+s}}^A\psi_A(X_{\tau+s}^{K_A})\dd s.$$
By Lemma \ref{cs}, since 
$|\cL^A_{\beta_{\tau+t}}\psi_A(x)|\leq C_A^{(4)}(1+\beta_{t+1})\indiq_{\{U_{K_A}(x)\geq D_A\}}$
and since $\beta_0\geq b_A\geq b^{(1)}_{K_A}$,
$$
\E\Big[\sup_{t\geq 0} |R_t|\Big] 
\leq C^{(4)}_A \int_0^\infty (1+\beta_{t+1})\PR(U_{K_A}(X_{\tau+t}^{K_A}) \geq D_A )\dd t
\leq \int_0^\infty \frac{C^{(3)}_A C^{(4)}_A (1+\beta_{t+1})}{(e^{c \beta_0} + t)^2}\dd t 
\leq \frac 1{40}.
$$
Consequently, for $E=\{\sup_{t\geq 0} |R_t|<1/10\}$, we have $\PR(E^c)\leq 1/4$.

\vip

On $E$, we have $M_t=M_{t\land\sigma}$, where
$\sigma=\inf\{t\geq 0 : M_t \notin [-1/10,11/10]\}$,
because $M_t+R_t=\psi_A(X_{\tau+t}^{K_A})$ takes values in $[0,1]$.
On
$\{\sup_{t\geq 0} U_{K_A}(X_{\tau+t}^{K_A})\geq K_A\}$, the process $\psi_A(X_{\tau+t}^{K_A})$ 
must up-cross
$[0,1]$ at least once, so that on $E\cap\{\sup_{t\geq 0} U_{K_A}(X_{\tau+t}^{K_A})\geq K_A\}$, 
the martingale
$M_t=M_{t\land\sigma}$ must up-cross $[1/10,9/10]$ at least once.
Hence 
$$
\PR\Big(\sup_{t\geq 0} U_{K_A}(X_{\tau+t}^{K_A})\geq K_A\Big) \leq \PR(E^c) + 
\PR(E,(M_{t\land\sigma})_{t\geq 0} \hbox{ up-crosses } [1/10,9/10]) \leq 1/4+p,
$$
where $p=\PR((M_{t\land\sigma})_{t\geq 0} \hbox{ up-crosses } [1/10,9/10])$.

\vip

By Doob's up-crossing inequality, see e.g. 
Revuz-Yor \cite[Proposition 2.1 page 61]{ry} we know that, 
for any continuous martingale $(Z_t)_{t\geq 0}$, any $a<b$, 
denoting by $U_{T,a,b}$ the number of up-crossings
of $[a,b]$ by $(Z_t)_{t\geq 0}$ during $[0,T]$, it holds 
that $(b-a)\E[U_{T,a,b}]\leq \E[(Z_T-a)_-]$.

\vip

Thus for $N_T$ the number of up-crossings of
$[1/10,9/10]$ by the martingale $(M_{t\land \sigma})_{t\geq 0}$ during $[0,T]$, it holds that
$$
p=\lim_{T\to\infty}\PR(N_T\geq 1)
\leq \lim_{T\to \infty} \E[N_T] \leq \lim_{T\to \infty} \frac{\E[(M_{T\land \sigma} - 1/10)_-]}{8/10}
\leq\frac{2/10}{8/10}=\frac14.
$$
We used that $M_{T\land \sigma}\geq -1/10$ by definition of $\sigma$.
We conclude that, for all $\beta>b_A$, we have $\PR(\sup_{t\geq 0} U_{K_A}(X_{\tau+t}^{K_A})\geq K_A) \leq 1/2$ as desired.
\end{proof}

\section{Success of the simulated annealing}\label{con}

We now show that no escape in large time implies the success of the simulated annealing.

\begin{proof}[Proof of Proposition \ref{mm}]
We assume $(A)$, fix $c>c_*$, $x_0\in \R^d$, $\beta_0 >0$ and consider the solution
$(X_t)_{t\geq 0}$ to \eqref{eds}. Since $\lim_{|x|\to \infty} U(x)=\infty$, our goal is to
show that for any fixed $\e>0$,
$$\lim_{t\to \infty}\PR\Big(\liminf_{s\to \infty} U(X_s)<\infty \hbox{ and } U(X_t)>\e\Big)=0.$$

{\it Step 1.} It suffices to show that for each $A\geq 1$, setting  
$$\Omega_A=\{\liminf_{s\to \infty} U(X_s)< A\},$$ 
it holds that
$\Omega_A\subset \{\sup_{s\geq 0} |X_s|<\infty\}$. 
Indeed, if this hold true, we fix $\eta>0$, consider $A_\eta>0$ large enough so that 
$\PR(A_\eta\leq \liminf_{s\to \infty} U(X_s)<\infty)<\eta$ and write
\begin{align*}
\PR\Big(\liminf_{s\to \infty} U(X_s)<\infty \hbox{ and } U(X_t)>\e\Big)
\leq& \eta + \PR(\Omega_{A_\eta} \hbox{ and } U(X_t)>\e)\\
\leq&  \eta+\PR\Big(\sup_{s\geq 0} |X_s|<\infty \hbox{ and } U(X_t)>\e\Big).
\end{align*}
Thus $\limsup_{t\to \infty}
\PR(\liminf_{s\to \infty} U(X_s)<\infty \hbox{ and } U(X_t)>\e)\leq \eta$ 
by Lemma \ref{fmt}. Since $\eta>0$ is arbitrarily small, the conclusion follows.

\vip

{\it Step 2.} We fix $A\geq 1$ and show that for $\Omega_A=\{\liminf_{s\to \infty} U(X_s)< A\}$, we have 
$\Omega_A\subset \{\sup_{s\geq 0} |X_s|<\infty\}$. 
\vip
We introduce $b_A>1$ and $K_A>A$ as in Proposition \ref{crucial} and 
consider $t_A\geq 0$ large enough so that $\beta_{t_A}\geq b_A$.
We set $S_0=t_A$ and, for all $k\geq 1$, 
$$T_{k}=\inf\{t>S_{k-1} : U(X_t)\leq A\} \quad\hbox{ and }\quad S_k=\inf\{t>T_k : U(X_t)\geq K_A\},$$ 
with the convention that $\inf \emptyset = \infty$.

\vip

We start from
\begin{align*}
\PR(S_{k+1}<\infty|S_k<\infty)=&\PR(T_{k+1}<\infty,S_{k+1}<\infty|S_k<\infty)\\
=&\E\Big[\indiq_{\{T_{k+1}<\infty\}} \PR(S_{k+1}<\infty |\cF_{T_{k+1}}) \Big|S_k<\infty\Big].
\end{align*}
But on $\{T_{k+1}<\infty\}$ and conditionally on $\cF_{T_{k+1}}$,
$(X_{T_{k+1}+t})_{t\geq 0}$ is a solution to \eqref{eds}, starting 
from $X_{T_{k+1}} \in \{U\leq A\}$,
with $\beta_0$ replaced by $\beta_{T_{k+1}}\geq \beta_{t_A}\geq b_A$. 
Hence, using Proposition \ref{crucial}, a.s.,
$$
\indiq_{\{T_{k+1}<\infty\}}\PR(S_{k+1}<\infty |\cF_{T_{k+1}})=\indiq_{\{T_{k+1}<\infty\}}
\PR\Big(\sup_{t\geq 0} U(X_{T_{k+1}+t})\geq K_A\Big)\leq 1/2.
$$
All this shows that for all $k\geq 1$, $\PR(S_{k+1}<\infty|S_k<\infty)\leq 1/2$. 

\vip

Consequently,
there a.s. exists $k\geq 1$ such that $S_k=\infty$, and we introduce
$$k_0=\inf\{k\geq 1 : S_k=\infty\}.$$ 
We then have $S_{k_0-1}<\infty=S_{k_0}$.
By definition of $\Omega_A$, it holds that $T_{k_0}<\infty$ on $\Omega_A$.
Since $U(X_t)< K_A$
for all $t\in [T_{k_0},S_{k_0})=[T_{k_0},\infty)$ (on $\Omega_A$), this implies that
$$\Omega_A\subset\Big\{\limsup_{s\to \infty} U(X_s)\leq K_A\Big\}
\subset \Big\{\sup_{s\geq 0} |X_s|<\infty\Big\}$$ 
as desired.
\end{proof}

We conclude the section with the

\begin{proof}[Proof of Theorem \ref{main}]
We assume $(A)$ and that there is $\alpha_0>0$ such that $\intrd e^{-\alpha_0 U(x)}\dd x <\infty$.
We fix $c>0$, $x_0 \in \R^d$, $\beta_0 >0$
and consider the unique solution $(X_t)_{t\geq 0}$
to \eqref{eds}.
By Proposition \ref{rec}, $\liminf_{t\to \infty} |X_t|<\infty$ a.s.
If moreover $c\!>\!c_*$, $\lim_{t\to \infty}U(X_t)\!=\! 0$ in probability by {Proposition \ref{mm}.}
\end{proof}

\section{Appendix: non-explosion}\label{imp3}

It remains to study the non-explosion of our process.
Surprisingly, this is rather tedious, except if assuming some Lyapunov condition,
for example that $-x\cdot \nabla U(x) \leq C(1+|x|^2)$, 
which forbids too nasty oscillations. We will prove the following result,
which is much stronger (but more natural) than what we really need,
since $U\geq 0$ under $(A)$.

\begin{thm}\label{nonex}
Assume that $U:\R^d \to \R$ and $\beta: \R_+\to (0,\infty)$
are of class $C^\infty$. Fix $x_0\in\R^d$ and 
consider the pathwise unique maximal solution $(X_t)_{t\in[0,\zeta)}$ to 
\begin{equation}\label{edsg}
X_t=x_0+B_t-\frac12\intot \beta_s\nabla U(X_s)\dd s, 
\end{equation}
where $\zeta=\lim_n \zeta_n$,
with $\zeta_n=\inf\{t\geq 0 : |X_t|\geq n\}$. Assume that
\begin{equation}\label{cu}
\hbox{there is $L>0$ such that for all $x\in \R^d$, }U(x)\geq -L(1+|x|^2).
\end{equation}
Then it holds that $\zeta=\infty$ a.s.
\end{thm}

Since $\nabla U$ is locally Lipschitz continuous, 
the existence of a pathwise unique possibly exploding solution is classical.
This result is rather natural: as is well-known, the solution to \eqref{edsg}, 
with $U(x)=-(1+|x|^2)^{\alpha}$ explodes if and only if $\alpha>1$. 
The difficulty relies in the fact that we do not want to assume any local property on $\nabla U$.
Let us mention that the proof below, assuming that $U\geq 0$, would be slightly simpler but less transparent.

\vip

Our proof is inspired by methods found in
Ichihara \cite{ichi}, who uses Dirichlet forms, and Grigor'yan \cite{gr1}
and \cite[Section 9]{gr2}, who studies manifold-valued diffusions. Both deal
with the time-homogeneous 
case ($\beta_t=\beta_0$ for all $t\geq 0$). 
In \cite{gr2}, non-explosion is proved under some
very weak conditions (allowing e.g. for some additional logarithmic factors in \eqref{cu}),
while \cite{ichi} is more stringent (roughly, he treats only the case where $U(x)\geq -L(1+|x|)$).

\vip

We start with the following remark.

\begin{rk}\label{rkne} 
(i) To prove Theorem \ref{nonex}, one may assume additionally that 
\begin{equation}\label{ep}
\hbox{there is $t_0>0$ such that $\beta_t=\beta_{t_0}$ for all $t\geq t_0$.}
\end{equation}

(ii) For any $x_0 \in \R^d$, for $(X_t)_{t\in[0,\zeta)}$ the solution to \eqref{edsg} and for
$t>0$, the measure $f_t$ defined by $f_{t}(A)=\PR(\zeta>t,X_t\in A)$ is absolutely continuous
with respect to the Lebesgue measure on $\R^d$.
\vip
(iii) It suffices to prove Theorem \ref{nonex} for a.e. $x_0\in \R^d$.
\end{rk}

\begin{proof} 
(i) Assume that Theorem \ref{nonex} holds under the additional condition \eqref{ep}
and consider $\beta:\R_+\to (0,\infty)$ of class $C^\infty$.
We fix $T>0$, introduce
$\bar \beta:\R_+\to (0,\infty)$ of class $C^\infty$ satisfying \eqref{ep},
such that $\beta_t=\bar \beta_t$ on $[0,T]$ and we introduce the corresponding solution
$(\bar X_t)_{t\in [0,\bar \zeta)}$. We have $(X_t)_{t\in [0,T\land \zeta)}=(\bar X_t)_{t\in [0,T\land \bar \zeta)}$,
whence in particular $\{\zeta\leq T\}=\{\bar \zeta\leq T\}$. Since $\bar \zeta=\infty$ a.s., we conclude
that $\PR(\zeta\leq T)=0$. Since $T$ is arbitrarily large, this implies
that $\zeta=\infty$ a.s.

\vip

(ii) Fix a Lebesgue-null set $A\in \R^d$.
Since $\nabla U$ is bounded on compact sets, we deduce from the Girsanov theorem that
$\PR(\zeta_n>t,X_t\in A)=0$ for all $n\geq 1$. By monotone
convergence, we conclude that $\PR(\zeta>t,X_t\in A)=0$ as desired.

\vip

(iii) Assume that for any $\beta: \R_+\to (0,\infty)$ of class $C^\infty$, $\PR_{0,x}(\zeta<\infty)=0$ for
a.e. $x\in\R^d$. Then for a given $\beta: \R_+\to (0,\infty)$ of class $C^\infty$, 
for all $t\geq 0$, $\PR_{t,x}(\zeta<\infty)=0$ for a.e. $x\in\R^d$. Since $\zeta>0$ a.s. by continuity,
we may write, for all $x_0\in\R^d$,  
$$\PR_{0,x_0}(\zeta<\infty)=\lim_{t\to 0}\PR_{0,x_0}(t<\zeta<\infty)= \lim_{t\to 0}
\E_{0,x_0}[\indiq_{\{\zeta>t\}}\PR_{t,X_t}(\zeta<\infty)]=0.
$$
We used the Markov property and that $\indiq_{\{\zeta>t\}}\PR_{t,X_t}(\zeta<\infty)=0$ a.s. when $t>0$ by point (ii)
and since $\PR_{t,x}(\zeta<\infty)=0$ for a.e. $x\in\R^d$.
\end{proof}

Above and in the whole section, we denote by $\E_{t_0,x_0}$ the expectation 
concerning the process starting from $x_0\in\R^d$ at time $t_0\geq 0$: 
under $\E_{t_0,x_0}$, the process $(X_t)_{t\geq 0}$ solves (in law) the S.D.E.
$X_t=x_0+B_t-\frac12\intot \beta_{t_0+s}\nabla U(X_s)\dd s$.

\vip
In the whole section, we denote by $v_d$ the volume of the unit ball and, for $r>0$, we set
$$
B_r=\{x\in\R^d : |x|<r\},\quad \bB_r=\{x\in\R^d :|x|\leq r\} \quad \hbox{and}\quad \dB_r=\{x\in\R^d :|x|=r\}.
$$
 
We will study of the following Kolmogorov backward equation, which consists of a particular
case of the Feynman-Kac formula.

\begin{lem}\label{back}
Adopt the assumptions of Theorem \ref{nonex} and suppose \eqref{ep}.
Fix $n\geq 1$ and $\alpha>0$. There is a function $u_{n,\alpha} \in C^{1,2}(\R_+\times \bB_n)$
such that $u_{n,\alpha}=1$ on $\R_+\times \dB_n$ and
\begin{equation}\label{bb}
\partial_t u_{n,\alpha}(t,x) + \cL_{\beta_t} u_{n,\alpha}(t,x)=\alpha u_{n,\alpha}(t,x) 
\quad \hbox{for } (t,x)\in [0,\infty)\times B_n.
\end{equation}
For $\varphi:\R^d\to \R$ of class $C^2$, $\beta>0$ and $x\in \R^d$,  
we have set 
$$
\cL_\beta \varphi(x)=\frac12 [\Delta \varphi(x) - \beta \nabla U(x)\cdot \nabla \varphi(x)].
$$
For any $t\geq 0$, any $x \in \bB_n$, it holds that $u_{n,\alpha}(t,x)=\E_{t,x}[\exp(-\alpha\zeta_{n})]$.
\end{lem}

\begin{proof} This relies one more time on classical results found in Friedman \cite{f}.
We fix some $t_0\geq 0$ such that $\beta_t=\beta_{t_0}$ for all $t\geq t_0$.
All the coefficients of \eqref{bb} are smooth and bounded, since restricted to $\bB_n$,
whose boundary is smooth. Hence all the results cited below do indeed apply.

\vip

By \cite[Chapter 3, Theorem 19]{f}, there exists a solution 
$v_{n,\alpha} \in C^2(\bB_n)$ to the elliptic boundary problem
$\cL_{\beta_{t_0}} v_{n,\alpha} = \alpha v_{n,\alpha}$   on $B_n$ and $v_{n,\alpha}=1$ on 
$\dB_n$.

\vip

By \cite[Chapter 3, Theorem 7]{f} (after time-reversing),
there exists a solution $w_{n,\alpha}$ belonging to ${C^{1,2}([0,t_0]\times \bB_n)}$ to the parabolic problem
$\partial_t  w_{n,\alpha}+\cL_{\beta_t}  w_{n,\alpha} =  \alpha w_{n,\alpha}$ on $(0,t_0)\times B_n$, 
with boundary condition $w_{n,\alpha}=1$ on $[0,t_0]\times \dB_n$ and terminal condition 
$w_{n,\alpha}(t_0,x)=v_{n,\alpha}(x)$ on $B_n$.

\vip

The function $u_{n,\alpha}$ defined by $u_{n,\alpha}(t,x)=v_{n,\alpha}(x)$ if $t\geq t_0$ and $u_{n,\alpha}(t,x)=w_{n,\alpha}(t,x)$ if $t\in [0,t_0]$
satisfies the conditions of the statement.

\vip

Finally, using the It\^o formula and \eqref{bb}, one checks that for all $t \geq 0$, all $x \in B_n$,
all $T\geq 0$, 
$$
\E_{t,x}[u_{n,\alpha}(T\land\zeta_{n},X_{T\land\zeta_n})e^{-\alpha(T\land\zeta_n)}]=u_{n,\alpha}(t,x).
$$
We let $T\to\infty$ and find that $\E_{t,x}[e^{-\alpha\zeta_n}]=u_{n,\alpha}(t,x)$ 
by dominated convergence and since $u_{n,\alpha}(\zeta_n,X_{\zeta_n})=1$ a.s.
\end{proof}

\begin{rk}\label{gr}
By the Green formula, for all $\beta>0$, for all $r>0$, for all $\varphi : \bB_r\to \R$ of class $C^1$
and $\psi : \bB_r\to \R$ of class $C^2$,
\begin{align*}
2\int_{B_r} \varphi(x) \cL_\beta \psi(x) e^{-\beta U(x)} \dd x 
=& -\int_{B_r} \nabla\varphi(x) \cdot
\nabla \psi(x) e^{-\beta U(x)} \dd x \\
&+ \int_{\dB_r} \varphi(x) [\nabla \psi(x)\cdot \nu(x)] e^{-\beta U(x)} \dd S,
\end{align*}
where $\nu(x)=x/|x|$ is the unit vector normal to $\dB_r$ and where
$\dd S$ is its surface element.
\end{rk} 

Although this is already known, see Grigor'yan \cite[Section 9]{gr2}, 
we recall for the sake of completeness
how to treat the homogeneous case. We use an approach closer to the one
of Ichihara \cite{ichi}, who however assumes more than \eqref{cu} and whose proof is
more intricate and relies on the study of
$v(x)=\E_x[e^{-\sigma_1}]$, where $\sigma_1=\inf\{t\geq 0 : |X_t|\leq 1\}$.

\begin{prop}\label{ich}
Assume that $U:\R^d\to \R$ is $C^\infty$ and satisfies \eqref{cu}. If $\beta_t=\beta_0>0$ for all $t\geq 0$,
then $\PR_x(\zeta<\infty)=0$ for a.e. $x \in \R^d$.
\end{prop}
\begin{proof}
For $\alpha>0$ and $n\geq 1$, we set $u_\alpha(x)=\E_{x}[e^{-\alpha \zeta}]$ and 
$u_{n,\alpha}(x)=\E_{x}[e^{-\alpha\zeta_n}]$. We divide the proof into 2 steps.
We recall that $L$ is defined in \eqref{cu}.

\vip

{\it Step 1.} Here we prove that for all $r>0$, there is a constant $C_r>0$ such that
\begin{equation}\label{rr1}
\forall \alpha>0,\quad \int_{B_r} u_{\alpha}^2(x)e^{-\beta_0U(x)}\dd x \leq C_r e^{-\alpha/(2\beta_0L)}.
\end{equation}

By Proposition
\ref{back}, $u_{n,\alpha} \in C^{2}(\bB_n)$,  $u_{n,\alpha}=1$ on $\dB_n$ and
$\cL_{\beta_0} u_{n,\alpha}=\alpha u_{n,\alpha}$ on $B_n$.
For any $r \in (0,n]$, we have
\begin{align*}
\Phi_{n,\alpha}(r):=&\int_{B_r} (2\alpha u_{n,\alpha}^2(x)+|\nabla u_{n,\alpha}(x)|^2)e^{-\beta_0U(x)}\dd x \\
=& \int_{\dB_r} u_{n,\alpha}(x)[\nabla u_{n,\alpha}(x)\cdot \nu(x)]e^{-\beta_0U(x)} \dd S.
\end{align*}
Indeed, it suffices to write $2\alpha u_{n,\alpha}^2=2 u_{n,\alpha}\cL_{\beta_0} u_{n,\alpha}$ 
and to use Remark \ref{gr} with $\varphi=\psi=u_{n,\alpha}$.
Hence for all $r \in (0,n]$, since $2\alpha a^2+b^2 \geq 2\sqrt{2\alpha }ab$,
\begin{align*}
\Phi_{n,\alpha}'(r)=&\int_{\dB_r} (2\alpha u_{n,\alpha}^2(x)+|\nabla u_{n,\alpha}(x)|^2)e^{-\beta_0U(x)}\dd S \\
\geq& 2\sqrt{2\alpha}\int_{\dB_r} u_{n,\alpha}(x)|\nabla u_{n,\alpha}(x)|e^{-\beta_0U(x)} \dd S
\geq 2\sqrt{2\alpha}\Phi_{n,\alpha}(r).
\end{align*}
Thus $\Phi_{n,\alpha}(r) \leq \Phi_{n,\alpha}(n)e^{2\sqrt{2\alpha}(r-n)}$ for all $r \in (0,n]$.

\vip
 
But, writing  $2\alpha u_{n,\alpha}=2 \cL_{\beta_0} u_{n,\alpha}$ and using Remark \ref{gr} 
with $\varphi=1$ and $\psi=u_{n,\alpha}$,
$$
2\int_{B_n} \alpha u_{n,\alpha}(x) e^{-\beta_0U(x)}\dd x = 
\int_{\dB_n} [\nabla u_{n,\alpha}(x)\cdot \nu(x)]e^{-\beta_0U(x)} \dd S
=\Phi_{n,\alpha}(n),
$$
because $u_{n,\alpha}=1$ on $\dB_n$. Hence $\Phi_{n,\alpha}(n)\leq 2\alpha \int_{B_n} e^{-\beta_0U(x)}\dd x 
\leq 2\alpha v_d n^d e^{\beta_0L(1+n^2)}$ by \eqref{cu}.
All this shows that for all $\alpha>0$, all $n\geq 1$, all 
$r\in (0,n]$,
\begin{equation*}
\int_{B_r} u_{n,\alpha}^2(x)e^{-\beta_0U(x)}\dd x\leq \frac{1}{2\alpha}\Phi_{n,\alpha}(r) \leq v_d n^d e^{\beta_0L(1+n^2)}
e^{2\sqrt{2\alpha}(r-n)}.
\end{equation*}
But $u_\alpha(x)\leq u_{n,\alpha}(x)$ for all $n\geq 1$, all $x\in \R^d$. 
Hence for $r>0$ fixed, with the choice $n=r+\sqrt \alpha/(\beta_0L)$,
we conclude that, for some constant $C_r>0$ depending on $r$ (and on $L$ and $d$), 
$$
\int_{B_r} u_{\alpha}^2(x)e^{-\beta_0U(x)}\dd x\leq v_d(r+\sqrt \alpha/(\beta_0L))^de^{\beta_0L(1+2r^2)+2\alpha/(\beta_0L)}
e^{-2\sqrt{2} \alpha /(\beta_0L) }\leq C_r e^{-\alpha/(2\beta_0 L)}.
$$

{\it Step 2.} We now conclude. Assume by contradiction that 
$\intrd\PR_x(\zeta<\infty)\dd x>0$ and fix $\eta>0$.
\vip

It holds that 
$\intrd\PR_x(\zeta\leq \eta)\dd x>0$. Else we would have, by the Markov property,
$$
\intrd \PR_x(\zeta\leq 2\eta)\dd x = \intrd \PR_x(\eta<\zeta\leq 2\eta)\dd x=
\intrd \E_x[\indiq_{\{\zeta>\eta\}}\PR_{X_\eta}(\zeta\leq \eta)]\dd x=0
$$
thanks to  Remark \ref{rkne}-(ii). Iterating the argument, we would find that
$\intrd \PR_x(\zeta\leq K\eta)\dd x=0$ for all $K\geq 1$, whence $\intrd\PR_x(\zeta<\infty)\dd x=0$.

\vip

Consequently, we can find $r_0>0$ such that $q:=\int_{B_{r_0}}[\PR_x(\zeta\leq\eta)]^2e^{-\beta_0 U(x)}\dd x>0$. 
We then have, 
since $u_\alpha(x)=\E_x[e^{-\alpha \zeta}]\geq e^{-\alpha\eta}\PR_x(\zeta\leq \eta)$,
\begin{align*}
\forall \alpha>0,\quad 
\int_{B_{r_0}} u_{\alpha}^2(x)e^{-\beta_0U(x)}\dd x \geq q e^{-2\alpha\eta}.
\end{align*}
With the choice $\eta= 1/(8\beta_0L)$, this contradicts \eqref{rr1}.
\end{proof}

Using in particular some clever ideas found in Grigor'yan \cite[Section 9]{gr2}, who
studies, in the homogeneous case, the P.D.E. satisfied by
$w(t,x)=\PR_{x}[\zeta<t]$, we can now give the

\begin{proof}[Proof of Theorem \ref{nonex}]
We consider $U:\R^d\to \R$ and $\beta:\R_+\to(0,\infty)$ of class $C^\infty$.
We recall that $\zeta=\lim_n \zeta_n$, with $\zeta_n=\inf\{t\geq 0 : |X_t|\geq n\}$.
We set $u_n(t,x)=\E_{t,x}[e^{-\zeta_n}]$ and $u(t,x)=\E_{t,x}[e^{-\zeta}]$, omitting
the subscript $\alpha$ since we now always work with $\alpha=1$.
By Remark \ref{rkne}, we may moreover suppose that there is $t_0>0$
such that $\beta_t=\beta_{t_0}$ for all $t\geq t_0$,
and it suffices to prove that $u(0,x)=0$ for a.e. $x\in\R^d$.
Since $u(t_0,x)=0$ for a.e. $x\in\R^d$ by Proposition \ref{ich}, it is sufficient to prove that
\begin{equation}\label{ooo}
\exists \delta_0>0, \forall t_1\in [0,t_0], 
\Big\{\intrd u(t_1,x)\dd x =0\Big\} \Rightarrow \Big\{\forall t\in [(t_1-\delta_0)\lor 0,t_1], \intrd u(t,x)\dd x
=0\Big\}.
\end{equation}

{\it Step 1.} Here we check that 
for all $a\in [0,1]$, all $b\geq 0$, all $\e>0$, all $\eta>0$,
$$
ab \leq \frac{a}{\eta \e} + \frac{e^{\eta b - 1/\e}}{\eta}.
$$
We fix $a\in [0,1]$, $\e>0$, and $\eta>0$ and study $f(b)=\frac{a}{\e} + e^{\eta b - 1/\e}-\eta ab$.
We have $f(0)>0$, $f(\infty)=\infty$ and $f'(b)=\eta[e^{\eta b - 1/\e}-a]$.
If $a\leq e^{-1/\e}$, then $f$ is non-decreasing, so that $f$ is nonnegative on $\R_+$. 
If now $a>e^{-1/\e}$,
then $f$ attains its minimum at $b_0=\frac1\eta[\log a+\frac1\e]$ and $f(b_0)=a-a\log a >0$.

\vip

{\it Step 2.} Here we prove that are some constants $\delta>0$ and $\kappa_0>0$ 
such that for all $t_1 \in [0,t_0]$, all $R> 1$, there is a $C^1$ function 
$\alpha_{t_1,R}:[(t_1-\delta)\lor 0,t_1]\times\R^d\to [0,1]$ enjoying the properties that
$$
\alpha_{t_1,R}= 1 \hbox{ on } [(t_1-\delta)\lor 0,t_1]\times B_R, \qquad 
\alpha_{t_1,R} = 0 \hbox{ on } [(t_1-\delta)\lor 0,t_1]\times B^c_{4R}
$$
and, for all $t\in [(t_1-\delta)\lor 0,t_1]$,
$$
m_{t_1,R}(t):=\intrd  \Big(|\nabla \alpha_{t_1,R}(t,x)|^2 -\partial_t [\alpha^2_{t_1,R}(t,x)]\Big)_+e^{-\beta_t U(x)}
\dd x  \leq \kappa_0 e^{-R^2}.
$$

We start with a $C^\infty$-function 
$\eta_R:\R^d \to [0,1]$ such that $\eta_R=1$ on $B_{2R}$, $\eta_R=0$ on
$B^c_{4R}$ and $|\nabla \eta_R|\leq 1/R$.
\vip
We also introduce the $C^1$ function on $[(t_1-\delta)\lor 0,t_1]\times\R^d$ defined by
$$
\xi_{t_1,R}(t,x)=\frac{(|x|-R)_+^2}{2(2\delta+t-t_1)}, 
$$
which equals $0$ on $[(t_1-\delta)\lor 0,t_1]\times B_R$ and solves 
$\partial_t \xi_{t_1,R} + \frac12|\nabla \xi_{t_1,R}|^2=0$ on $[(t_1-\delta)\lor 0,t_1]\times B^c_R$.
\vip
If $\delta>0$ is small enough, the function $\alpha_{t_1,R}(t,x)=\eta_R(x)\exp(-\xi_{t_1,R}(t,x)/2)$ 
enjoys the desired properties:
it is $C^1$, $[0,1]$-valued, we have $\alpha_{t_1,R}(t,x)=1$ if $|x|\leq R$ and $\alpha_R(t,x)=0$ if $|x|\geq 4R$,
and we have
\begin{align*}
|\nabla \alpha_{t_1,R}|^2 -\partial_t [\alpha^2_{t_1,R}]
=& \Big|\nabla \eta_{t_1,R} -\frac12 \eta_{t_1,R} \nabla\xi_{t_1,R}\Big|^2 e^{-\xi_{t_1,R}}  + \eta_{t_1,R}^2 e^{-\xi_{t_1,R}}  \partial_t \xi_{t_1,R}\\
\leq& \Big(2 |\nabla \eta_{t_1,R}|^2 + \frac12\eta_{t_1,R}^2 |\nabla\xi_{t_1,R}|^2 + \eta_{t_1,R}^2\partial_t \xi_{t_1,R}\Big)e^{-\xi_{t_1,R}}
= 2 |\nabla \eta_{t_1,R}|^2 e^{-\xi_{t_1,R}}.
\end{align*}
Since $|\nabla \eta_{t_1,R}|\leq R^{-1}\indiq_{B_{4R}\setminus B_{2R}}$ and since
$\xi_{t_1,R}\geq \frac{R^2}{4\delta}$ on $[(t_1-\delta)\lor 0,t_1]\times B^c_{2R}$,
we deduce from \eqref{cu}, that 
\begin{align*}
m_{t_1,R}(t)\leq& 2R^{-2} e^{-R^2/(4\delta)}e^{\beta_t L(1+16R^2)} \hbox{ Vol} (B_{4R}\setminus B_{2R}).
\end{align*}
Since $\beta$ is bounded on $[0,t_0]$, it indeed suffices to choose $\delta>0$ small enough to complete the step.

\vip

{\it Step 3.} We consider $\delta>0$ as in Step 2, fix $t_1 \in [0,t_0]$ and set,
for $R>1$, $n\geq 5R$ and $t\in [(t_1-\delta)\lor 0,t_1]$, 
$$
\varphi_{n,t_1,R}(t)=\intrd u_n^2(t,x)\alpha_{t_1,R}^2(t,x)e^{-\beta_t U(x)}\dd x.
$$
The goal of this step is to verify that there is a constant $\kappa_1>0$ such that
\begin{equation}\label{tbtb}
\forall t_1 \in [0,t_0], \; \forall R>1, \; \forall n\geq 5R, 
\; \forall t\in [(t_1-\delta)\lor 0,t_1], 
\;\;\varphi_{n,t_1,R}'(t) \geq -\kappa_1\Big[ R^2 \varphi_{n,t_1,R}(t) + e^{-R^2}\Big].
\end{equation}

By \eqref{bb} (with $\alpha=1$), 
we know that $\partial_t u_n(t,x)=u_n(t,x)-\cL_{\beta_t}u_n(t,x)$ on $[0,\infty)\times B_n$.
Since Supp$\;\alpha_{t_1,R}(t,\cdot) \subset B_{4R}$ and since $n\geq 5R$, we may write
$$
\varphi_{n,t_1,R}'(t)=I_{n,t_1,R}(t)+J_{n,t_1,R}(t)-K_{n,t_1,R}(t), 
$$
where
\begin{align*}
I_{n,t_1,R}(t)=&\intrd 2[u_n(t,x)-\cL_{\beta_t}u_n(t,x)]u_n(t,x)\alpha^2_{t_1,R}(t,x)e^{-\beta_t U(x)}\dd x,\\
J_{n,t_1,R}(t)=&\intrd u_n^2(t,x)\partial_t [\alpha^2_{t_1,R}(t,x)]e^{-\beta_t U(x)}\dd x,\\
K_{n,t_1,R}(t)=&\beta'_t\intrd U(x)u_n^2(t,x)\alpha_{t_1,R}^2(t,x)e^{-\beta_t U(x)}\dd x.
\end{align*}

Using Remark \ref{gr} with any $r>4R$, since $\alpha_{t_1,R}(t,\cdot)$ 
is supported in $B_{4R}$, we have
\begin{align*}
I_{n,t_1,R}(t)\geq& - 2 \intrd \cL_{\beta_t}u_n(t,x) [u_n(t,x)\alpha^2_{t_1,R}(t,x)]e^{-\beta_t U(x)}\dd x\\
=& \intrd \nabla u_n(t,x) \cdot \nabla [u_n(t,x)\alpha_{t_1,R}^2(t,x)]e^{-\beta_t U(x)}\dd x\\
=& \intrd \Big[|\nabla u_n(t,x)|^2\alpha_{t_1,R}^2(t,x)+ 2 u_n(t,x)\alpha_{t_1,R}(t,x)\nabla u_n(t,x)\cdot \nabla
\alpha_{t_1,R}(t,x)\Big] e^{-\beta_t U(x)}\dd x\\
\geq& - \intrd u_n^2(t,x) |\nabla \alpha_{t_1,R}(t,x)|^2 e^{-\beta_t U(x)}\dd x.
\end{align*}
We finally used that $a^2-2ab\geq-b^2$ with $a=|\nabla u_n|\alpha_{t_1,R}$ and 
and $b= u_n |\nabla \alpha_{t_1,R}|$. Thus
\begin{align*}
I_{n,t_1,R}(t)+J_{n,t_1,R}(t)\geq & \intrd u_n^2(t,x) \Big(\partial_t [\alpha^2_{t_1,R}(t,x)]
- |\nabla \alpha_{t_1,R}(t,x)|^2 \Big) e^{-\beta_t U(x)}\dd x\\
\geq & - \intrd  \Big( |\nabla \alpha_{t_1,R}(t,x)|^2-\partial_t [\alpha^2_{t_1,R}(t,x)]\Big)_+ 
e^{-\beta_t U(x)}\dd x \geq - \kappa_0 e^{-R^2}
\end{align*}
by Step 2. We used that $u_n^2(t,x)\in [0,1]$.

\vip
We next write $K_{n,t_1,R}(t)=K^{(1)}_{n,t_1,R}(t)+K^{(2)}_{n,t_1,R}(t)$, where
$$
K^{(1)}_{n,t_1,R}(t)=\beta'_t\intrd \indiq_{\{U(x)\leq 0\}}U(x)u_n^2(t,x)\alpha_{t_1,R}^2(t,x)e^{-\beta_t U(x)}\dd x
\leq C R^2 \varphi_{n,t_1,R}(t),
$$
because $|\beta'|$ is bounded on $[0,t_0]$ and because $|\indiq_{\{U\leq 0\}}U|\leq L(1+(4R)^2)$ 
on Supp$\;\alpha_{t_1,R}(t,\cdot) \subset B_{4R}$ 
by \eqref{cu}, and
$$
K^{(2)}_{n,t_1,R}(t)=\beta'_t\intrd \indiq_{\{U(x)\geq 0\}}U(x)u_n^2(t,x)\alpha_{t_1,R}^2(t,x)e^{-\beta_t U(x)}\dd x.
$$
By Step 1 with $a=u_n^2(t,x)\alpha_{t_1,R}^2(t,x)\in [0,1]$, $b=U(x)\geq 0$,
$\eta=\beta_t>0$ and $\e=(2R)^{-2}$,
\begin{align*}
K^{(2)}_{n,t_1,R} \leq& \frac{|\beta'_t|}{\beta_t} \int_{B_{4R}} \Big[2R^2 u_n^2(t,x)\alpha_{t_1,R}^2(t,x)
+ e^{\beta_t U(x)-2R^2}\Big]e^{-\beta_t U(x)}\dd x\\
=&  \frac{|\beta'_t|}{\beta_t} \Big[2R^2 \varphi_{n,t_1,R}(t) + \hbox{Vol}(B_{4R})e^{-2R^2} \Big]\\
\leq & C \Big[R^2 \varphi_{n,t_1,R}(t) + e^{-R^2} \Big],
\end{align*}
since $|\beta'_t|/\beta_t$ is bounded on $[0,t_0]$. This ends the step.

\vip

{\it Step 4.} We now conclude that \eqref{ooo} holds true with $\delta_0=\min\{\delta,1/(2\kappa_1)\}$, where
$\delta>0$ and $\kappa_1>0$ were introduced in Steps 2 and 3.
We thus fix $t_1 \in [0,t_0]$ and assume that $\intrd u(t_1,x)\dd x=0$.
\vip

Integrating \eqref{tbtb}, we find that for all $t \in [(t_1-\delta_0)\lor 0,t_1]$, all
$R>1$ and all $n\geq 5R$,
$$
\varphi_{n,t_1,R}(t) \leq \varphi_{n,t_1,R}(t_1)e^{\kappa_1 R^2(t_1-t)}+ R^{-2}e^{-R^2}[e^{\kappa_1 R^2(t_1-t)}-1]
\leq e^{R^2/2}\varphi_{n,t_1,R}(t_1)+e^{-R^2/2},
$$
the last inequality following from the fact that $t_1-t\leq\delta_0\leq1/(2\kappa_1)$.

\vip
Since $\lim_n u_n(t,x)=u(t,x)$ by dominated convergence and since $\alpha_{t_1,R}(t,\cdot)$ is compactly supported,
we have $\lim_{n} \varphi_{n,t_1,R}(t)= \varphi_{t_1,R}(t)$, where
$\varphi_R(t)= \intrd u^2(t,x)\alpha_{t_1,R}^2(t,x)e^{-\beta_t U(x)}\dd x$.
We thus find, for all $t \in [(t_1-\delta_0)\lor 0,t_1]$, all $R>1$
$$
\varphi_{t_1,R}(t) \leq  e^{R^2/2}\varphi_{t_1,R}(t_1)+e^{-R^2/2}.
$$
But since $u(t_1,x)=0$ for a.e. $x\in \R^d$, it holds that
$\varphi_{t_1,R}(t_1)=0$ for all $R>1$. Hence for all fixed $t \in [(t_1-\delta_0)\lor 0,t_1]$, 
all fixed $R_0>0$, all $R>R_0>1$, since $\alpha_{t_1,R}(t,\cdot)\geq\indiq_{B_{R_0}}$,
$$
\int_{B_{R_0}} u^2(t,x)e^{-\beta_t U(x)}\dd x \leq \varphi_{t_1,R}(t) \leq e^{-R^2/2},
$$
whence $\int_{B_{R_0}} u^2(t,x)e^{-\beta_t U(x)}\dd x=0$ and thus $u(t,x)=0$ for a.e. $x\in\R^d$ as desired.
\end{proof}

\end{document}